\documentclass{amsart}
\usepackage{amssymb}
\usepackage{mathrsfs}
\usepackage{stmaryrd}
\usepackage[all]{xy}

\newcommand{\uuline}[1]{\underline{\underline{#1}}}

\newcommand{\C}{\mathbb{C}}
\newcommand{\Ql}{\mathbb{Q}_\ell}
\newcommand{\Z}{{\mathbb{Z}}}

\newcommand{\K}{\mathbb{K}}
\newcommand{\F}{\mathbb{F}}
\renewcommand{\O}{\mathbb{O}}
\newcommand{\E}{\mathbb{E}}
\newcommand{\bk}{\Bbbk}

\newcommand{\cA}{\mathcal{A}}
\newcommand{\gr}{{\mathrm{gr}}}
\newcommand{\Lgr}{L^\gr}
\newcommand{\dgr}{\Delta^\gr}
\newcommand{\ngr}{\nabla^\gr}
\newcommand{\Gr}{\mathrm{Gr}}

\newcommand{\ngmod}{\mathsf{mod}}
\newcommand{\gmod}{\mathsf{gmod}}
\newcommand{\lh}{\text{-}}

\newcommand{\cT}{\mathcal{T}}
\newcommand{\Db}{D^{\mathrm{b}}}
\newcommand{\Kb}{K^{\mathrm{b}}}

\newcommand{\sD}{\mathsf{D}}

\newcommand{\scS}{\mathscr{S}}
\newcommand{\scT}{\mathscr{T}}
\newcommand{\Parity}{\mathsf{Parity}}
\newcommand{\even}{{\mathrm{ev}}}
\newcommand{\For}{\mathsf{For}}

\newcommand{\mix}{\mathrm{mix}}
\newcommand{\Dmix}{D^\mix}
\newcommand{\Perv}{\mathsf{P}}
\newcommand{\dmix}{\Delta^{\mix}}
\newcommand{\nmix}{\nabla^{\mix}}
\newcommand{\uuE}{\uuline{\E}{}}
\newcommand{\uuO}{\uuline{\O}{}}
\newcommand{\cE}{\mathcal{E}}
\newcommand{\D}{\mathbb{D}}
\newcommand{\p}{{}^p\!}
\newcommand{\pH}{\p\mathcal{H}}
\newcommand{\ptau}{\p\tau}
\newcommand{\IC}{\mathcal{IC}}
\newcommand{\cP}{\mathcal{P}}
\newcommand{\cF}{\mathcal{F}}
\newcommand{\cG}{\mathcal{G}}

\newcommand{\Projf}{\mathsf{Projf}}
\newcommand{\Tilt}{\mathsf{Tilt}}
\newcommand{\Radon}{\mathsf{R}}
\newcommand{\VV}{\mathbb{V}}
\newcommand{\cO}{\mathcal{O}}

\newcommand{\cB}{\mathscr{B}}
\newcommand{\scP}{\mathscr{P}}

\newcommand{\Gv}{\check{G}}
\newcommand{\Bv}{\check{B}}
\newcommand{\Tv}{\check{T}}
\newcommand{\cBv}{\check{\cB}}

\newcommand{\dv}{\check{\Delta}}
\newcommand{\nv}{\check{\nabla}}
\newcommand{\cTv}{\check{\cT}}
\newcommand{\cPv}{\check{\cP}}
\newcommand{\cEv}{\check{\cE}}
\newcommand{\muv}{{\check \mu}}
\newcommand{\nuv}{{\check \nu}}

\newcommand{\simto}{\xrightarrow{\sim}}
\newcommand{\la}{\langle}
\newcommand{\ra}{\rangle}
\newcommand{\op}{{\mathrm{op}}}
\DeclareMathOperator{\End}{End}
\DeclareMathOperator{\Hom}{Hom}
\DeclareMathOperator{\uHom}{\underline{Hom}}
\DeclareMathOperator{\RHom}{\mathit{R}Hom}
\DeclareMathOperator{\Ext}{Ext}
\DeclareMathOperator{\Tor}{Tor}
\DeclareMathOperator{\Irr}{Irr}
\newcommand{\id}{\mathrm{id}}
\DeclareMathOperator{\im}{im}

\makeatletter
\def\lotimes{\@ifnextchar_{\@lotimessub}{\@lotimesnosub}}
\def\@lotimessub_#1{\mathchoice{\mathbin{\mathop{\otimes}^L}_{#1}}%
  {\otimes^L_{#1}}{\otimes^L_{#1}}{\otimes^L_{#1}}}
\def\@lotimesnosub{\mathbin{\mathop{\otimes}^L}}
\makeatother
\newcommand{\tboxtimes}{\mathbin{\widetilde{\mathord{\boxtimes}}}}

\newtheorem*{thm*}{Theorem}
\numberwithin{equation}{section}
\newtheorem{thm}{Theorem}[section]
\newtheorem{lem}[thm]{Lemma}
\newtheorem{prop}[thm]{Proposition}
\newtheorem{cor}[thm]{Corollary}
\theoremstyle{definition}
\newtheorem{defn}[thm]{Definition}
\theoremstyle{remark}
\newtheorem{rmk}[thm]{Remark}

\title[Modular perverse sheaves on flag varieties II]{Modular perverse sheaves on flag varieties II:\\ Koszul Duality and Formality}

\author{Pramod N. Achar}
\address{Department of Mathematics\\
  Louisiana State University\\
  Baton Rouge, LA 70803\\
  U.S.A.}
\email{pramod@math.lsu.edu}

\author{Simon Riche}
\address{Universit{\'e} Blaise Pascal - Clermont-Ferrand II, Laboratoire de Math{\'e}matiques, CNRS, UMR 6620, Campus universitaire des C{\'e}zeaux, F-63177 Aubi{\`e}re Cedex, France
}
\email{simon.riche@math.univ-bpclermont.fr}

\thanks{P.A. was supported by NSF Grant No.~DMS-1001594.  S.R. was supported by ANR Grants No.~ANR-09-JCJC-0102-01, ANR-2010-BLAN-110-02 and ANR-13-BS01-0001-01.}

\begin{document}

\begin{abstract}
Building on the theory of parity sheaves due to Juteau--Mautner--Williamson, we develop a formalism of ``mixed modular perverse sheaves'' for varieties equipped with a stratification by affine spaces.  We then give two applications: (1)~a ``Koszul-type'' derived equivalence relating a given flag variety to the Langlands dual flag variety, and (2)~a formality theorem for the modular derived category of a flag variety (extending the main result of \cite{rsw}). 
\end{abstract}

\maketitle

\section{Introduction}
\label{sec:intro}

\subsection{}

This paper continues the study of modular perverse sheaves on flag varieties begun in~\cite{ar}.   We retain the notation and conventions of~\cite[\S1.2 and \S1.8]{ar}. In particular, $G$ denotes a connected complex reductive group, $\cB$ its flag variety, and $\Db_{(B)}(\cB,\E)$ the derived category of complexes of $\E$-sheaves that are constructible with respect to the orbits of a fixed Borel subgroup $B$.  Here, $\E$ is any member of an $\ell$-modular system $(\K, \O, \F)$ (see Section~\ref{sec:mixed-der}), where $\ell$ is a good prime for $G$.

A summary of seven motivating properties of $\Db_{(B)}(\cB,\C)$ appeared in~\cite[\S1.3]{ar}.  In this paper, we study modular versions of items (4)~(``Koszul duality''), (5)~(``self-duality''), and (7)~(``formality'').

\subsection{}
\label{ss:mixed-sheaves}

The main new tool in the present paper is a theory of ``mixed modular perverse sheaves.''  For perverse $\Ql$-sheaves on a variety defined over a finite field, the term ``mixed'' usually means: ``Pay attention to the eigenvalues of the Frobenius action on stalks.''  The additional structure obtained in this way has profound consequences, thanks largely to Deligne's reformulation of the Weil conjectures~\cite{deligne}.  A number of important results in representation theory make essential use of deep properties of mixed $\Ql$-sheaves; for examples, see~\cite{abg, bgs, bezru2}.  For sheaves with coefficients in $\O$ or $\F$, it still makes sense to consider the Frobenius action (and this was done in~\cite{rsw}), but without an analogue of the Weil conjectures, it becomes a much more difficult notion to work with.  

In this paper, we propose a new approach to defining the word ``mixed'' in the modular setting.  This approach does not involve varieties over finite fields or Galois actions in any way.  Instead, we will build a category from scratch that bears many of the hallmarks of~\cite{deligne,bbd}, such as a ``Tate twist.'' (It also has a theory of ``weights'' and ``purity''; these are studied systematically in~\cite{ar3}.)  The approach we take is philosophically quite close to that of~\cite{ar:kdsf}. It involves in a crucial way the \emph{parity sheaves} of~\cite{jmw}. For the flag variety $\cB$, this new category, denoted by $\Dmix_{(B)}(\cB,\E)$, is the main object of study in this paper\footnote{For some choices of variety, this category can also be given a ``combinatorial'' description, by replacing the language of parity sheaves by that of Soergel (bi)modules or sheaves on moment graphs.  This perspective reveals that $\Dmix_{(B)}(\cB,\E)$ has antecedents in the literature:  for instance, the homotopy category of Soergel bimodules, which appears in Rouquier's categorification of the braid group~\cite{rouquier}, is an incarnation of the equivariant mixed modular derived category $\Dmix_B(\cB,\F)$.}.

\subsection{Self-duality}
\label{ss:intro-self-duality}

As an application, we prove the following analogue of the char\-ac\-ter\-istic-zero ``self-duality'' theorem of Bezrukavnikov--Yun~\cite[Theorem~5.3.1]{by}. In this statement, $\Gv$ is the Langlands dual reductive group, $\Bv \subset \Gv$ is a Borel subgroup, and $\cBv=\Gv/\Bv$ is the flag variety of $\Gv$.

\begin{thm*}[Self-duality]
There is an equivalence of triangulated categories $\kappa: \Dmix_{(B)}(\cB,\E) \simto \Dmix_{(\Bv)}(\cBv,\E)$ that swaps parity sheaves and tilting perverse sheaves.
\end{thm*}

See Theorem~\ref{thm:self-duality} for a more precise statement. The nomenclature of this result refers to the fact that $\kappa$ is symmetric; the two triangulated categories appearing in the statement are defined in the same way (in contrast with the other derived equivalences in~\cite{by}), and both $\kappa$ and $\kappa^{-1}$ send mixed parity sheaves to mixed tilting perverse sheaves.  In~\cite{by}, this result was called ``Koszul self-duality.''  (The role of the term ``Koszul'' will be discussed further in~\S\ref{ss:intro-positivity}.)

Simultaneously with the proof of 
the theorem,
we will construct a t-exact functor $\mu: \Dmix_{(B)}(\cB,\E) \to \Db_{(B)}(\cB,\E)$ that makes $\Dmix_{(B)}(\cB,\E)$ into a ``graded version'' of $\Db_{(B)}(\cB,\E)$.  We will also work out the behavior of the usual classes of objects (simple, standard, tilting, parity) under each of the functors in the diagram
\[
\Db_{(B)}(\cB,\E) \xleftarrow{\mu} \Dmix_{(B)}(\cB,\E) \xrightarrow[\sim]{\kappa} \Dmix_{(\Bv)}(\cBv,\E) \xrightarrow{\muv} \Db_{(\Bv)}(\cBv,\E).
\]
This picture is a modular analogue of~\cite[\S1.3(4)]{ar}.

\subsection{Formality}
\label{ss:intro-formality}

General homological arguments show that many triangulated categories can be described in terms of dg-modules over some dg-algebra.  It is a far more subtle and delicate problem to decide whether the
dg-algebra in question is \emph{formal}---i.e., quasi-isomorphic to a graded ring with zero differential.  One typical argument involves equipping the dg-algebra with an additional ``internal'' grading.  In the standard proof of formality for $\Db_{(B)}(\cB,\C)$, this additional grading comes from the Frobenius action discussed in~\S\ref{ss:mixed-sheaves}.

In~\cite{rsw}, these methods were extended to the study of $\Db_{(B)}(\cB,\F)$.  However, because the Frobenius endomorphism is itself an $\F$-linear operator, care must be taken with the characteristic of $\F$ to ensure that eigenspace decompositions behave well.  In~\cite{rsw}, $\Db_{(B)}(\cB,\F)$ and $\Db_{(B)}(\cB,\O)$ were shown to be formal when $\ell > 2 \dim \cB + 1$.  In the present paper, by using $\Dmix_{(B)}(\cB,\E)$ in place of a Frobenius action, we establish formality in all good characteristics. (See Section~\ref{sec:formality} for details on notation, and Theorem~\ref{thm:formality} for a more precise statement.)

\begin{thm*}[Formality]
Let $\cE$ be the direct sum of the indecomposable parity complexes $\cE_w(\E)$, and set $\mathbf{E}:=\Hom^\bullet_{\Db_{(B)}(\cB,\E)}(\cE,\cE)$. Then there exists an equivalence of triangulated categories $\Db_{(B)}(\cB,\E) \cong \mathsf{dgDerf}\lh\mathbf{E}$, where $\mathbf{E}$ is considered as a dg-algebra with trivial differential.
\end{thm*}

We expect this to hold even when $\ell$ is bad for $G$, but the proof we give here relies on the results of~\cite{ar}, which are not (yet?) available in bad characteristic.

\subsection{Towards positivity}
\label{ss:intro-positivity}

In the formality theorem for $\Db_{(B)}(\cB,\C)$, the ring $\mathbf{E}$ that arises is positively graded and Koszul.  This means, in particular, that its degree-zero component is semisimple.  It is now known that this last assertion cannot hold for general $\F$: if it did, it would follow that the parity sheaves on $\cB$ coincide with the simple perverse sheaves, but counterexamples have been found by Braden~\cite[Appendix~A]{wilbraden} and, more recently, by Williamson~\cite{williamson}.

It may be reasonable to ask only that $\mathbf{E}$ be positively graded (but not necessarily Koszul).  This is equivalent to asking that all parity sheaves on $\cB$ be perverse (but not necessarily simple).  In a subsequent paper~\cite{ar3}, we will study various conditions that imply or are implied by the positivity of the grading.  As of this writing, there are no known counterexamples in good characteristic to the positivity of the grading on $\mathbf{E}$.

\subsection{Contents}

The foundations of the mixed derived category and its perverse t-structure are developed in Sections~\ref{sec:mixed-der} and~\ref{sec:mixed-perv}.  In Section~\ref{sec:kac-moody} we prove that the (partial) flag varieties of Kac--Moody groups satisfy the assumptions needed for the theory of Sections~\ref{sec:mixed-der} and~\ref{sec:mixed-perv} to apply.
The main results (as stated in \S\S\ref{ss:intro-self-duality}--\ref{ss:intro-formality}) are proved in Section~\ref{sec:formality}. 
Finally, appendix~\ref{sec:homological} contains a brief review of definitions and facts about graded quasihereditary categories.

\section{The mixed derived category}
\label{sec:mixed-der}

Let $\ell$ be a prime number, and let $(\K, \O, \F)$ be an $\ell$-modular system (i.e.~$\K$ is a finite extension of $\Ql$, $\O$ is its ring of integers, and $\F$ is the residue field of $\O$). We use the letter $\E$ to denote any member of $(\K,\O,\F)$. In this section, we do not impose any constraint on $\ell$. We denote by $\E\lh\ngmod$, resp.~$\E\lh\gmod$, the category of finitely generated $\E$-modules, resp.~finitely generated graded $\E$-modules.

\subsection{Varieties and sheaves}
\label{ss:var}

Let $X$ be a complex algebraic variety equipped with a fixed finite algebraic stratification
\[
X=\bigsqcup_{s \in \scS} X_s
\]
in which each $X_s$ is isomorphic to an affine space.  We denote by $i_s: X_s \hookrightarrow X$ the inclusion map. Let $\Db_\scS(X,\E)$ denote the derived category of $\E$-sheaves on $X$ (in the analytic topology) that are constructible with respect to the given stratification.  In a minor abuse of notation, if $Y \subset X$ is a locally closed union of strata, we will also use the letter $\scS$ to refer to the induced stratification of $Y$.  (For instance, we will write $\Db_\scS(Y,\E)$ rather than $\Db_{\{s \in \scS \mid X_s \subset Y\}}(Y,\E)$.)  Throughout the paper, the shift functor on $\Db_\scS(X,\E)$, usually written as $[1]$, will instead be denoted by
\[
\{1\}: \Db_\scS(X,\E) \to \Db_\scS(X,\E).
\]
(The reason for this nonstandard notation will become clear below.)  Given a finitely generated $\E$-module $M$, let $\underline{M}{}_{X_s}$ be the constant sheaf with value $M$ on $X_s$.  We will refer often to the perverse sheaf $\underline{M}_{X_s}\{\dim X_s\}$, and so we introduce the notation
\[
\uuline{M}{}_{X_s} := \underline{M}{}_{X_s} \{ \dim X_s \}.
\]
Here, and throughout the paper, the notation ``$\dim$'' applied to a variety should always be understood to mean \emph{complex} dimension.

All varieties in the paper will be assumed to satisfy the following condition:
\begin{enumerate}
\item[\bf(A1)] For each $s \in \scS$, there is an indecomposable parity complex $\cE_s(\E) \in \Db_\scS(X,\E)$ that is supported on $\overline{X_s}$ and satisfies $i_s^*\cE_s(\E) \cong \uuE{}_{X_s}$.
\end{enumerate}
(See~\cite[Definition~2.4]{jmw} for the definition of parity complexes. Here and below, the term ``parity'' refers to the constant pariversity denoted $\natural$ in \cite{jmw}.)
This is only an additional hypothesis when $\E = \O$; for $\E = \K$ or $\F$, it holds automatically by~\cite[Corollary~2.28]{jmw}.  When $\E = \O$, ~\cite[Corollary~2.35]{jmw} gives a sufficient condition for this to hold. According to~\cite[Theorem~4.6]{jmw}, generalized flag varieties (with the Bruhat stratification) satisfy this hypothesis.  In any case, if $\cE_s(\E)$ exists, it is unique up to isomorphism: see~\cite[Theorem~2.12]{jmw}.  When there is no risk of ambiguity, we may simply call this object $\cE_s$.  

We denote by $\Parity_\scS(X,\E)$ the full additive subcategory of $\Db_\scS(X,\E)$ consisting of parity complexes.  Every object in $\Parity_\scS(X,\E)$ is isomorphic to a direct sum of objects of the form $\cE_s(\E)\{n\}$~\cite[Theorem~2.12]{jmw}.  Note that the shift functor $\{1\}$ restricts to an autoequivalence $\{1\} : \Parity_\scS(X,\E) \to \Parity_\scS(X,\E)$.

It will occasionally be useful to refer to the full subcategory $\Parity^\even_\scS(X,\E) \subset \Parity_\scS(X,\E)$ consisting of \emph{even} objects---that is, direct sums of $\cE_s\{n\}$ with $n \equiv \dim X_s \pmod 2$.  Objects of $\Parity^\even_\scS(X,\E)\{1\}$ are said to be \emph{odd}.  It follows from~\cite[Corollary~2.8]{jmw} that
\begin{equation}\label{eqn:parity-decompose}
\Parity_\scS(X,\E) \cong \Parity^\even_\scS(X,\E) \oplus \Parity^\even_\scS(X,\E)\{1\}.
\end{equation}
In particular, if $\cF$ is even and $\cG$ is odd, then $\Hom(\cF,\cG) = \Hom(\cG,\cF) = 0$.

\subsection{The mixed derived category}
\label{ss:Dmix}

Our main object of study will be the following triangulated category:
\[
\Dmix_\scS(X,\E):=\Kb \Parity_\scS(X,\E).
\]
The shift functor for this category is denoted by $[1]: \Dmix_\scS(X,\E) \to \Dmix_\scS(X,\E)$.  This is different from the autoequivalence $\{1\}: \Dmix_\scS(X,\E) \to \Dmix_\scS(X,\E)$ that it inherits from $\Parity_\scS(X,\E)$.  It will be convenient to introduce the notation 
\[
\langle n \rangle := \{-n\}[n].
\]
The functor $\la 1\ra$ is called the \emph{Tate twist}.
We can of course regard the $\cE_s(\E)$ as objects of $\Dmix_\scS(X,\E)$, via the obvious embedding $\Parity_\scS(X,\E) \hookrightarrow \Dmix_\scS(X,\E)$. Nevertheless, it will often serve us well to pay explicit attention to the ambient category, so we introduce the additional notation
\[
\cE^\mix_s(\E) := \text{$\cE_s(\E)$, regarded as an object of $\Dmix_\scS(X,\E)$.}
\]

\begin{defn}
The category $\Dmix_\scS(X,\E)$ is called the \emph{mixed derived category} of $X$ with coefficients in $\E$.
\end{defn}

As the terminology indicates, this is intended to be a kind of replacement for the mixed derived category of~\cite{bbd}, although there are two salient differences.  First, the definition of $\Dmix_\scS(X,\E)$ does not involve any Frobenius action; rather, the ``additional grading'' provided by the Frobenius action in~\cite{bbd} is replaced here by the ``internal shift'' $\la1\ra$ in $\Dmix_{\scS}(X,\E)$.

Second, in general, \emph{there is no obvious functor from $\Dmix_\scS(X,\E)$ to $\Db_\scS(X,\E)$.}  We will eventually construct such a functor for flag varieties, but that construction relies heavily on the results of~\cite{ar}.  The existence of such a functor is very closely related to the formality theorem.

\begin{rmk}
When $\E = \K$, the situation is somewhat better.  Under some additional hypotheses on the simple perverse $\K$-sheaves on $X$, our category $\Dmix_\scS(X,\K)$ is equivalent to the category introduced in~\cite[\S 7.2]{ar:kdsf}. The theory developed in~\cite{ar:kdsf} gives a functor $\Dmix_\scS(X,\K) \to \Db_\scS(X,\K)$ as part of a broader picture relating $\Dmix_\scS(X,\K)$ to the mixed sheaves of~\cite{bbd}. This theory will not be used in the present paper, however.
\end{rmk}

Note that the decomposition~\eqref{eqn:parity-decompose} implies a similar decomposition for the mixed derived category:
\begin{equation}\label{eqn:dmix-decompose}
\Dmix_\scS(X,\E) \cong \Kb\Parity^\even_\scS(X,\E) \oplus \Kb(\Parity^\even_\scS(X,\E)\{1\}).
\end{equation}
An object of $\Kb\Parity^\even_\scS(X,\E)$ (resp.~$\Kb(\Parity^\even_\scS(X,\E)\{1\})$) is said to be \emph{even} (resp.~\emph{odd}).  Any indecomposable object of $\Dmix_\scS(X,\E)$ must be either even or odd.

\subsection{Some functors}
\label{ss:functors}

The Verdier duality functor $\D_X : \Db_{\scS}(X,\E) \simto \Db_{\scS}(X,\E)$ restricts to an antiequivalence of $\Parity_\scS(X,\E)$ (see~\cite[Remark~2.5(3)]{jmw}), which then induces an antiequivalence
\[
\D_X : \Dmix_{\scS}(X,\E) \simto \Dmix_\scS(X,\E).
\]
This functor satisfies 
\[
\D_X \circ \{ n \} \cong \{ -n \} \circ \D_X, \quad \D_X \circ [ n ] \cong [ -n ] \circ \D_X, \quad \D_X \circ \langle n \rangle \cong \langle -n \rangle \circ \D_X.
\]

We will denote by
\[
\K(-) : \Db_{\scS}(X,\O) \to \Db_\scS(X,\K) \qquad \text{and} \qquad \F(-) : \Db_\scS(X,\O) \to \Db_\scS(X,\F)
\]
the functors of (derived) extension of scalars. These functors send parity complexes to parity complexes~\cite[Lemma 2.37]{jmw}, so they also define functors
\[
\K(-) : \Dmix_{\scS}(X,\O) \to \Dmix_\scS(X,\K) \qquad \text{and} \qquad \F(-) : \Dmix_\scS(X,\O) \to \Dmix_\scS(X,\F).
\]

Finally, let $i : Z \hookrightarrow X$ and $j : U \hookrightarrow X$ be closed and open inclusions of unions of strata, respectively.  Then the functors $i_*$ and  $j^*$ restrict to the categories of parity complexes, and then define functors
\[
i_* : \Dmix_{\scS}(Z,\E) \to \Dmix_\scS(X,\E), \qquad j^* : \Dmix_\scS(X,\E) \to \Dmix_\scS(U,\E).
\]
These functors commute with the functors $\F(-)$ and $\K(-)$, and with Verdier duality. Note also that the functor $i_*$ is fully faithful. We will often use this functor to identify $\Dmix_{\scS}(Z,\E)$ with a full subcategory of $\Dmix_{\scS}(X,\E)$.

\subsection{Adjoints}
\label{ss:adjoints}

The goal of this subsection is to prove that the categories of the form $\Dmix_\scS(X,\E)$ can be endowed with a ``recollement'' structure in the sense of \cite[\S 1.4]{bbd}.  We fix an open union of strata $U$, and denote by $Z$ its complement. We denote by $j : U \hookrightarrow X$ and $i : Z \hookrightarrow X$ the inclusions. We clearly have $j^* i_* = 0$.

\begin{prop}
\label{prop:recollement}
The functor $j^* : \Dmix_{\scS}(X,\E) \to \Dmix_{\scS}(U,\E)$ admits a left adjoint $j_{(!)}$ and a right adjoint $j_{(*)}$.  Similarly, the functor $i_* : \Dmix_\scS(Z,\E) \to \Dmix_\scS(X,\E)$ admits a left adjoint $i^{(*)}$ and a right adjoint $i^{(!)}$.  Together, these functors give a recollement diagram
\[
\xymatrix@C=1.5cm{
\Dmix_\scS(Z,\E) \ar[r]^{i_*} &
\Dmix_\scS(X,\E) \ar[r]^{j^*} 
  \ar@/_1pc/[l]_{i^{(*)}} \ar@/^1pc/[l]^{i^{(!)}} &
\Dmix_\scS(U,\E). 
  \ar@/_1pc/[l]_{j_{(!)}} \ar@/^1pc/[l]^{j_{(*)}}
}
\]
\end{prop}
The extra parentheses in the names of the functors $j_{(!)}$, $j_{(*)}$, $i^{(*)}$, and $i^{(!)}$ are there to help us distinguish them from the usual functors $j_!$, $j_*$, $i^*$, and $i^!$ involving the ordinary derived category $\Db_\scS(X,\E)$.  Because all eight of these functors appear in the arguments below, we will maintain this distinction through the end of Section~\ref{sec:mixed-der}.  Starting from Section~\ref{sec:mixed-perv}, however, we will drop the extra parentheses in the new functors in Proposition~\ref{prop:recollement}.

Before proving Proposition~\ref{prop:recollement} in full generality, we consider the case in which $Z=X_s$ is a closed stratum.

\begin{lem}
\label{lem:recollement-one-stratum}
Let $X_s \subset X$ be a closed stratum.
\begin{enumerate}
\item
\label{it:recollement-X_s}
In the case $Z=X_s$, $j^*$ admits a left adjoint $j_{(!)}$ and a right adjoint $j_{(*)}$, such that the adjunction morphisms $j^* j_{(*)} \to \id$ and $\id \to j^* j_{(!)}$ are isomorphisms, and such that we have
\begin{align*}
\Dmix_\scS(X,\E) & = j_{(!)} \bigl( \Dmix_\scS(U,\E) \bigr) \ast i_* \bigl( \Dmix_\scS(Z,\E) \bigr) \\
\Dmix_\scS(X,\E) & = i_* \bigl( \Dmix_\scS(Z,\E) \bigr) \ast j_{(*)} \bigl( \Dmix_\scS(U,\E) \bigr)
\end{align*}
in the notation of~\cite[\S1.3.9]{bbd}.
\item
\label{it:recollement-X_s-2}
If $Z \subset X$ is a closed union of strata containing $X_s$, and if $j : X \smallsetminus X_s \hookrightarrow X$, $j_Z : Z \smallsetminus X_s \hookrightarrow Z$, $k : Z \hookrightarrow X$, $k_Z : Z \smallsetminus X_s \hookrightarrow X \smallsetminus X_s$ denote the inclusions, the functors $j_{(!)}$, $j_{(*)}$, $j_{Z(!)}$, $j_{Z(*)}$ given by \eqref{it:recollement-X_s} satisfy
\[
j_{(!)} k_{Z*} \cong k_* j_{Z(!)}, \qquad j_{(*)} k_{Z*} \cong k_* j_{Z(*)}.
\]
\end{enumerate}
\end{lem}

\begin{proof}
We treat the case of $j_{(!)}$ in detail; the case of $j_{(*)}$ is similar, or can be deduced using Verdier duality.

First we remark that we have $\cE_s=i_* \uuE_{X_s}$, and that the functor $i^*$ restricts to a functor from $\Parity_{\scS}(X,\E)$ to $\Parity_{\scS}(X_s,\E)$. For any $t \in \scS \smallsetminus \{s\}$ we denote by $\cE_t^+$ the image of the complex
\[
\cE_t \to i_{*} i_{}^* \cE_t
\]
(in degrees $0$ and $1$, and where the morphism is provided by adjunction)
in the category $\Dmix_\scS(X,\E)=\Kb \Parity_\scS(X,\E)$. 
Note that for any $n,m \in \Z$ we have
\begin{equation}
\label{eqn:E+}
\Hom_{\Dmix_\scS(X,\E)}(\cE_t^+, \cE_s\{m\}[n])=0.
\end{equation}
Indeed, this property follows from the observation that the natural morphism
\[
\Hom(i_{*} i^* \cE_t, \cE_s\{m\}[n]) \to
\Hom(\cE_t, \cE_s\{m\}[n])
\]
is an isomorphism for any $n,m \in \Z$, using the long exact sequence associated with the natural distinguished triangle
\[
i_{s*} i_s^* \cE_t[-1] \to \cE_t^+ \to \cE_t \xrightarrow{[1]}
\]
in $\Dmix_\scS(X,\E)$.

Let $\sD^+$ be the triangulated subcategory of $\Dmix_\scS(X,\E)$ generated by the objects $\cE_t^+\{m\}$ for all $t \in \scS \smallsetminus \{s\}$ and $m \in \Z$, and let $\iota : \sD^+ \to \Dmix_\scS(X,\E)$ be the inclusion. We claim that for any $\cF^+$ in $\sD^+$ and $\cG$ in $\Dmix_\scS(X,\E)$ the morphism
\begin{equation}
\label{eqn:morphisms-j^*}
\Hom_{\Dmix_\scS(X,\E)}(\iota \cF^+, \cG) \to \Hom_{\Dmix_\scS(U,\E)}(j^* \iota \cF^+, j^* \cG)
\end{equation}
induced by $j^*$ is an isomorphism. Indeed, by standard arguments using the five-lemma it is enough to prove the result when $\cF^+=\cE_t^+$ for some $t \in \scS \smallsetminus \{s\}$ and $\cG=\cE_u\{m\}[n]$ for some $u \in \scS$ and $n,m \in \Z$. If $u=s$ then the result follows from \eqref{eqn:E+}. Now assume $u \neq s$. If $n \notin \{-1,0\}$ there is nothing to prove. Assume now that $n=-1$. Then the right-hand side of~\eqref{eqn:morphisms-j^*} is zero, and the left-hand side consists of morphisms
\[
\varphi : i_{*} i^* \cE_t \to \cE_u\{m\}
\]
whose composition with the adjunction morphism $\cE_t \to i_{*} i^* \cE_t$ is zero. If $m$ does not have the same parity as $\dim X_u -\dim X_t$ then $\varphi=0$ by \cite[Corollary 2.8]{jmw}. Now if $m$ has the same parity as $\dim X_u -\dim X_t$ then $\varphi$, considered as a morphism in $\Db_\scS(X,\E)$, factors through a morphism $j_!j^*\cE_t\{1\} \to \cE_u\{m\}$. But
\[
\Hom_{\Db_\scS(X,\E)}(j_!j^*\cE_t\{1\}, \cE_u\{m\}) \cong \Hom_{\Db_\scS(U,\E)}(j^*\cE_t, j^*\cE_u\{m-1\})=0,
\]
again by \cite[Corollary 2.8]{jmw}, and this implies that $\varphi=0$. Finally, assume that $n=0$. If $m$ does not have the same parity as $\dim X_u -\dim X_t$, then both sides of \eqref{eqn:morphisms-j^*} vanish and there is nothing to prove. Otherwise the left-hand side of \eqref{eqn:morphisms-j^*} is the quotient of $\Hom(\cE_t,\cE_u\{m\})$ by the image of $\Hom(i_{*} i^* \cE_t, \cE_u\{m\})$. This space can easily be identified with the right-hand side of \eqref{eqn:morphisms-j^*} using the long exact sequence associated with the distinguished triangle
\[
j_! j^* \cE_t \to \cE_t \to i_{*} i^* \cE_t \xrightarrow{\{1\}}
\]
in $\Db_\scS(X,\E)$ and again \cite[Corollary 2.8]{jmw}.

Our claim regarding \eqref{eqn:morphisms-j^*} tells us in particular that $j^* \circ \iota$ is fully faithful. As the objects $j^* \iota \cE_t^+ \{m\} = j^* \cE_t \{m\}$ generate the triangulated category $\Dmix_\scS(U,\E)$, it follows that this functor is an equivalence of categories.
Now we define the functor 
\[
j_{(!)} := \iota \circ (j^* \circ \iota)^{-1} : \Dmix_\scS(U,\E) \to \Dmix_\scS(X,\E).
\]
By definition we have a natural isomorphism $j^* j_{(!)} \cong \mathrm{id}$. To prove that $j_{(!)}$ is left-adjoint to $j^*$ we have to prove that the morphism
\[
\Hom_{\Dmix_\scS(X,\E)}(j_{(!)} \cF, \cG) \xrightarrow{j^*} \Hom_{\Dmix_\scS(U,\E)}(j^* j_{(!)} \cF, j^* \cG)
\cong \Hom_{\Dmix_\scS(U,\E)}(\cF,j^* \cG)
\]
is an isomorphism for any $\cF \in \Dmix_\scS(U,\E)$ and $\cG \in \Dmix_\scS(X,\E)$. 
This follows from the observation that \eqref{eqn:morphisms-j^*} is an isomorphism.

By construction $j_{(!)} \cE_t=\cE_t^+$, and hence it is clear that the category $\Dmix_\scS(X,\E)$ is generated by the essential images of the functors $j_{(!)}$ and $i_*$. Since there exists no nonzero morphism from an object of $j_{(!)} \bigl( \Dmix_\scS(U, \E) \bigr)$ to an object of $i_{*} \bigl( \Dmix_\scS(Z,\E) \bigr)$ (by adjunction, and since $j^* i_*=0$), the first equality in~\eqref{it:recollement-X_s} follows.

Now we turn to \eqref{it:recollement-X_s-2}.
One can consider the full subcategories $\mathsf{D}^+_X \subset D^\mix_{\scS}(X,\E)$ and $\mathsf{D}^+_Z \subset D^\mix_{\scS}(Z,\E)$ constructed as in the proof of~\eqref{it:recollement-X_s}, and the inclusions $\iota_X$ and $\iota_Z$. It follows from the definitions that there exists a unique functor $k_{*}^+ : \mathsf{D}^+_Z \to \mathsf{D}^+_X$ such that $\iota_X \circ k_{*}^+ = k_{*} \circ \iota_Z$. Moreover, we have  $j^* \circ k_* = k_{Z*} \circ j_Z^*$, from which we deduce that
$j^* \circ \iota_X \circ k_{*}^+ = k_{Z*} \circ j_Z^* \circ \iota_Z$. Composing on the left with $j_{(!)}=\iota_X \circ (j^* \circ \iota_X)^{-1}$ and on the right with $(j_Z^* \circ \iota_Z)^{-1}$, we obtain the first isomorphism in \eqref{it:recollement-X_s-2}.
\end{proof}

\begin{proof}[Proof of Proposition~{\rm \ref{prop:recollement}}]
We will explain how to construct $j_{(!)}$ and $i^{(*)}$; we will prove that the adjunction morphisms $\id \to j^*j_{(!)}$ and $i^{(*)}i_* \to \id$ are isomorphisms; and we will show that for any $\cF \in \Dmix_\scS(X,\E)$, there is a morphism $i_* i^{(*)} \cF \to j_{(!)} j^* \cF [1]$ such that the triangle
\begin{equation}\label{eqn:triangles-U-Z-pre}
j_{(!)} j^* \cF \to \cF \to i_* i^{(*)} \cF \xrightarrow{[1]}
\end{equation}
is distinguished.  The axioms for a recollement (see~\cite[\S 1.4.3]{bbd}) consist of these assertions together with parallel ones for $j_{(*)}$ and $i^{(!)}$, and the condition that $j^* i_* = 0$.  The proofs for $j_{(*)}$ and $i^{(!)}$ are similar to those for $j_{(!)}$ and $i^{(*)}$, and will be omitted.  

We will first show, by induction of the number of strata in $Z$, that $j_{(!)}$ exists, that the adjunction morphism $\id \to j^*j_{(!)}$ is an isomorphism, and that we have
\begin{equation}\label{eqn:triangles-U-Z}
\Dmix_\scS(X,\E) = j_{(!)} \bigl( \Dmix_\scS(U,\E) \bigr) \ast i_* \bigl( \Dmix_\scS(Z,\E) \bigr).
\end{equation}

If $Z$ consists of one stratum, our assertions are proved in Lemma~\ref{lem:recollement-one-stratum}\eqref{it:recollement-X_s}.
Now assume that $Z$ has more than one stratum. Let $X_s \subset Z$ be a closed stratum, and set $X':=X \smallsetminus X_s$, $Z':=Z \smallsetminus X_s$. By induction the restriction functors associated with the inclusions $j' : U \hookrightarrow X'$ and $j'' : X' \hookrightarrow X$ have left adjoints; hence the same holds for their composition, which is the functor $j^*$. Similarly, the fact that $j^* j_{(!)} \cong \mathrm{id}$ follows from induction. By induction \eqref{eqn:triangles-U-Z} holds for the decompositions $X=X' \sqcup X_s$ and $X'=U \sqcup Z'$. Using associativity of the ``$\ast$'' operation (see \cite[Lemme 1.3.10]{bbd}) and the first isomorphism in Lemma~\ref{lem:recollement-one-stratum}\eqref{it:recollement-X_s-2} we deduce that it also holds for the decomposition $X=U \sqcup Z$. This finishes the induction.

Now, let us construct $i^{(*)}$ and prove the existence of~\eqref{eqn:triangles-U-Z-pre}. By \eqref{eqn:triangles-U-Z} and the fact that $i_*$ and $j_{(!)}$ are fully faithful, for any $\cF$ in $\Dmix_\scS(X,\E)$ there exist unique objects $\cF'$ in $\Dmix_\scS(U,\E)$ and $\cF''$ in $\Dmix_\scS(Z,\E)$ and a unique distinguished triangle
\begin{equation}
\label{eqn:triangle-def-i^*}
j_{(!)} \cF' \to \cF \to i_* \cF'' \xrightarrow{[1]}.
\end{equation}
(Unicity follows from~\cite[Corollaire~1.1.10]{bbd}.)
We necessarily have $\cF' \cong j^* \cF$, and we set $i^{(*)} \cF := \cF''$. Then the expected properties of the functor $i^{(*)}$ are clear by construction.
\end{proof}

We have observed in~\S\ref{ss:functors} that in the setting of Proposition~\ref{prop:recollement}, there exist natural isomorphisms $j^* \circ \D_X \cong \D_U \circ j^*$ and $i_* \circ \D_Z \cong \D_X \circ i_*$. We deduce canonical isomorphisms
\begin{equation}
\label{eqn:D!*}
\D_X \circ j_{(!)} \cong j_{(*)} \circ \D_U, \qquad \D_Z \circ i^{(!)} \cong i^{(*)} \circ \D_X.
\end{equation}

\begin{prop}
\label{prop:commutation-K-F}
In the setting of Proposition~{\rm \ref{prop:recollement}}, the functors $i_*$, $i^{(*)}$, $i^{(!)}$, $j^*$, $j_{(*)}$, and $j_{(!)}$ commute with the functors $\K(-)$ and $\F(-)$.
\end{prop}

\begin{proof}
This statement has already been observed in~\S\ref{ss:functors} in the case of $i_*$ and $j^*$. We prove it for $j_{(!)}$ and $i^{(*)}$; the proof for $j_{(*)}$ and $i^{(!)}$ is similar.

To prove the claim for $j_{(!)}$, it is enough to treat the case where
$Z$ contains only one (closed) stratum $X_s$. In this case,
the subcategory $\sD_{\E}^+$ introduced in the proof of Lemma~\ref{lem:recollement-one-stratum} is generated (as a triangulated category) by all complexes
\[
\cE \to i_*i^* \cE
\]
for $\cE$ in $\Parity_{\scS}(X,\E)$. (Indeed, if $\cE=\cE_s\{m\}$ for some $m \in \Z$ this complex is homotopic to $0$.) Hence we have
\[
\K(\sD^+_{\O}) \subset \sD^+_{\K}, \qquad \F(\sD^+_{\O}) \subset \sD^+_{\F}.
\]
Since the functors $j^*$ and $\iota$ commute with the functors $\K(-)$ and $\F(-)$,
we deduce the commutativity for the functors $j_{(!)}$. 

Finally, to prove the claim for the functors $i^{(*)}$ we observe that the formation of triangle \eqref{eqn:triangle-def-i^*} commutes with $\K(-)$ and $\F(-)$ in the obvious sense, which implies the desired commutativity.
\end{proof}

\subsection{Locally closed inclusions}
\label{ss:loc-closed}

The previous section dealt with open and closed inclusions.  In this section, we construct pullback and push-forward functors for any locally closed inclusion $h: Y \hookrightarrow X$.

\begin{lem}
\label{lem:functors}
Consider a commutative diagram
\[
\xymatrix{
W \ar@{^{(}->}[r]^-{i'} \ar@{^{(}->}[d]_-{j'} & Z \ar@{^{(}->}[d]^-{j} \\
Y \ar@{^{(}->}[r]^-{i} & X
}
\]
where $W$, $Z$ and $Y$ are unions of strata, and assume that $i$, $i'$ are closed embeddings and $j$, $j'$ are open embeddings.
Then there exist natural isomorphisms of functors
\[
j_{(!)} i'_* \cong i_* j'_{(!)}, \quad j_{(*)} i'_* \cong i_* j'_{(*)}, \quad i'{}^{(!)} j^* \cong j'{}^* i^{(!)}, \quad i'{}^{(*)} j^* \cong j'{}^* i^{(*)}.
\]
\end{lem}

\begin{proof}
It is enough to prove the first isomorphism: then the second one follows using Verdier duality (see~\eqref{eqn:D!*}), and the third by adjunction. Finally, the fourth isomorphism follows from the second one by adjunction. We will prove this claim by induction on the number of strata in $X$.

If $Z=Y=X$ then $W$ is a union of connected components of $X$, and the claim is easily checked.

Now assume that $Z \neq X$, and choose a closed stratum $X_s$ contained in $X \smallsetminus Z$. Set $X':=X \smallsetminus X_s$, $Y':=Y \cap X'$. Then one can complete our diagram to
\begin{equation}
\label{eqn:diagram-loc-closed}
\vcenter{
\xymatrix@R=0.5cm{
W \ar@{^{(}->}[r] \ar@{^{(}->}[d] & Z \ar@{^{(}->}[d] \\
Y' \ar@{^{(}->}[r] \ar@{^{(}->}[d] & X' \ar@{^{(}->}[d] \\
Y \ar@{^{(}->}[r] & X.
}
}
\end{equation}
By induction the claim is known for the upper part of the diagram, so that it is enough to prove it for the lower part. If $X_s \subset Y$ then this claim was proved in Lemma~\ref{lem:recollement-one-stratum}\eqref{it:recollement-X_s-2}. If $X_s \not\subset Y$, then $Y'=Y$, and a similar argument applies.

Finally, consider the case $Z=X$ but $Y \neq X$. If $W=Y$ there is nothing to prove. Otherwise we choose a closed stratum $X_s \subset Y$ not included in $W$, and set $X':=X \smallsetminus X_s$, $Y':=Y \smallsetminus X_s$. Then we have a diagram as in~\eqref{eqn:diagram-loc-closed} (with $Z$ replaced by $X'$ in the upper right corner), and we conclude using induction, Lemma~\ref{lem:recollement-one-stratum}\eqref{it:recollement-X_s-2} and the preceding cases.
\end{proof}

Lemma~\ref{lem:functors} allows us to define unambiguously, for any locally closed inclusion of strata $h : Y \to X$, the functors
\[
h_{(!)}, h_{(*)} : \Dmix_\scS(Y,\E) \to \Dmix_\scS(X,\E), \qquad h^{(*)}, h^{(!)} : \Dmix_\scS(X,\E) \to \Dmix_\scS(Y,\E).
\]
Let us explain the case of $h_{(!)}$. We first observe that if  $h=i \circ j$ with $i : Z \hookrightarrow X$ a closed inclusion and $j : Y \hookrightarrow Z$ an open inclusion, the functor $i_* \circ j_{(!)}$ does not depend on the choice of $Z$ (up to isomorphism). Indeed, applying Lemma~\ref{lem:functors} to the commutative diagram
\[
\xymatrix{
Y \ar@{=}[r] \ar@{^{(}->}[d]_-{j_Y} & Y \ar@{^{(}->}[d]^-{j} \\
\overline{Y} \ar@{^{(}->}[r]^-{i_Y} & Z
}
\]
we obtain that $j_{(!)} \cong i_{Y*} j_{Y(!)}$, and then $i_* \circ j_{(!)} \cong (i \circ i_Y)_* j_{Y(!)}$, which does not depend on the choice of $Z$. A similar argument using $X \smallsetminus (\overline{Y} \smallsetminus Y)$ shows that if $h=j \circ i$ with $i : Y \hookrightarrow U$ a closed inclusion and $j : U \hookrightarrow X$ an open inclusion, then the functor $j_{(!)} i_*$ does not depend on the choice of $U$ (up to isomorphism). Finally, another application of Lemma~\ref{lem:functors} tells us that all these functors are isomorphic to each other; they define the functor $h_{(!)}$.

One can easily check (using again Lemma~\ref{lem:functors}) that if $h : Y \hookrightarrow Z$ and $k : Z \hookrightarrow X$ are locally closed inclusions then we have
\[
(k \circ h)_{(!)} \cong k_{(!)} h_{(!)}, \ (k \circ h)_{(*)} \cong k_{(*)} h_{(*)}, \ (k \circ h)^{(!)} \cong h^{(!)} k^{(!)}, \ (k \circ h)^{(*)} \cong h^{(*)} k^{(*)}.
\]
Moreover, using \eqref{eqn:D!*} we have
\begin{equation}
\label{eqn:D!*-2}
\D_X \circ h_{(!)} \cong h_{(*)} \circ \D_Y, \qquad \D_Y \circ h^{(!)} \cong h^{(*)} \circ \D_X.
\end{equation}
Finally, using Proposition~\ref{prop:commutation-K-F} we have
\begin{equation}\label{eqn:inclusions-scalars}
\begin{aligned}
\F \circ h_{(!)} \cong h_{(!)} \circ \F, & \qquad \F \circ h_{(*)} \cong h_{(*)} \circ \F, \\
\K \circ h_{(!)} \cong h_{(!)} \circ \K, & \qquad \K \circ h_{(*)} \cong h_{(*)} \circ \K,
\end{aligned}
\end{equation}
and similarly for $h^{(!)}$ and $h^{(*)}$.

\begin{rmk}
\label{rmk:restriction-X_s}
Assume $Y=X_s$ is a stratum, so that $h=i_s$. The functor $i_s^* : \Db_{\scS}(X,\E) \to \Db_{\scS}(X_s,\E)$ restricts to a functor $\Parity_{\scS}(X,\E) \to \Parity_{\scS}(X_s,\E)$, which itself defines functor $i_s^{[*]} : \Dmix_\scS(X,\E) \to \Dmix_\scS(X_s,\E)$ between bounded homotopy categories. We claim that this functor is isomorphic to $i_s^{(*)}$. Indeed, by construction of $i_s^{(*)}$ one can assume that $X_s$ is closed in $X$. In this case $i_{s(*)}$ is simply the functor induced by $i_{s*}$, considered as a functor $\Parity_{\scS}(X_s,\E) \to \Parity_{\scS}(X,\E)$. The latter functor is right adjoint both to $i_s^{[*]}$ and $i_s^{(*)}$, which implies that these functors are isomorphic. In particular, we deduce that
\[
\Hom_{\Dmix_\scS(X,\E)}(\cE,i_{s(*)} \cF \{n\}[m] ) \cong 
\begin{cases}
\Hom_{\Db_{\scS}(X_s,\E)}(i_s^* \cE, \cF\{n\}) & \text{if $m=0$;} \\
0 & \text{otherwise}
\end{cases}
\]
for any $\cE$ in $\Parity_{\scS}(X,\E)$ and $\cF$ in $\Parity_{\scS}(X_s,\E)$.

Similarly, one can check that the functor $i_s^{(!)}$ is the functor induced by the restriction of $i_s^!$ to $\Parity_{\scS}(X,\E)$, and deduce an explicit description of the $\E$-module $\Hom_{\Dmix_\scS(X,\E)}(i_{s(!)} \cF, \cE\{n\}[m])$ for $\cE$ in $\Parity_{\scS}(X,\E)$ and $\cF$ in $\Parity_{\scS}(X_s,\E)$.
\end{rmk}

\subsection{Stratified morphisms}
\label{ss:stratified}

Let $Y = \bigsqcup_{t \in \scT} Y_t$ be another variety stratified by affine spaces and satisfying~{\bf (A1)}.  
We will say that a map $f: X \to Y$ is \emph{stratified} if the following two conditions hold:
\begin{enumerate}
\item For each $t \in \scT$, $f^{-1}(Y_t)$ is a union of strata.
\item For each $X_s \subset f^{-1}(Y_t)$, the map $f_{st}: X_s \to Y_t$ induced by $f$ is a trivial fibration, i.e.~we have $\dim X_s \geq \dim Y_t$, and $f_{st}$ is identified with the natural projection $\mathbb{A}^{\dim X_s} \to \mathbb{A}^{\dim Y_t}$.
\end{enumerate}
Note that these conditions imply that $f$ is stratified in the sense of~\cite[Definition~2.32]{jmw}, and that it is an even morphism in the sense of~\cite[Definition~2.33]{jmw}.

According to~\cite[Proposition~2.34]{jmw}, if $f$ is stratified and proper, then $f_! = f_*$ sends parity complexes to parity complexes, and thus induces a functor
\[
f_! = f_*: \Dmix_\scS(X,\E) \to \Dmix_\scT(Y,\E).
\]
On the other hand, if $f$ is stratified and smooth of relative dimension $d$, it is easy to see from the definitions that $f^* \cong f^!\{-2d\}$ sends parity complexes to parity complexes as well, and hence also induces a functor $\Dmix_\scT(Y,\E) \to \Dmix_\scS(X,\E)$.  Let
\[
f^\dag := f^*\{d\} \cong f^!\{-d\}: \Dmix_\scT(Y,\E) \to \Dmix_\scS(X,\E).
\]
This functor has the advantage that it commutes with Verdier duality.

If $f$ is both proper and smooth, then the functors above inherit the adjunction properties of the corresponding functors between $\Db_\scS(X,\E)$ and $\Db_\scT(Y,\E)$.  In particular, if we let
\[
f_\dag := f_*\{d\}
\qquad\text{and}\qquad
f_\ddag := f_*\{-d\},
\]
then $f_\dag$ is left adjoint to $f^\dag$, and $f_\ddag$ is right adjoint to $f^\dag$.

\begin{prop}\label{prop:proper-smooth}
Let $f: X \to Y$ be a proper, smooth stratified morphism.  Let $h: Z \to Y$ be the inclusion of a locally closed union of strata, and form the diagram
\[
\xymatrix{
f^{-1}(Z) \ar[r]^-{h'} \ar[d]_{f'} & X \ar[d]^{f} \\
Z \ar[r]_h & Y.}
\]
Then $f'$ is proper, smooth and stratified, and we have the following natural isomorphisms of functors:
\begin{align}
f_* \circ h'_{(*)} &\cong h_{(*)} \circ f'_*, & \label{eqn:ps-push}
f_* \circ h'_{(!)} &\cong h_{(!)} \circ f'_*, \\
h^{\prime(*)} \circ f^\dag & \cong f^{\prime\dag} \circ h^{(*)}, & \label{eqn:ps-pull}
h^{\prime(!)} \circ f^\dag & \cong f^{\prime\dag} \circ h^{(!)}, \\
f^\dag \circ h_{(*)} &\cong h'_{(*)} \circ f^{\prime\dag}, & \label{eqn:ps-bc-inc}
f^\dag \circ h_{(!)} &\cong h'_{(!)} \circ f^{\prime\dag}, \\
h^{(*)} \circ f_* & \cong f'_* \circ h^{\prime(*)}, & \label{eqn:ps-bc}
h^{(!)} \circ f_* & \cong f'_* \circ h^{\prime(!)}.
\end{align}
\end{prop}
\begin{proof}
By the considerations of~\S\ref{ss:loc-closed}, it is enough to treat the cases where $h$ is either an open inclusion or a closed inclusion.  If $h$ is a closed inclusion, then~\eqref{eqn:ps-push} follows from the fact that the functors
\[
f_* \circ h'_*, \ h_* \circ f_* : \Parity_\scS(f^{-1}(Z),\E) \to \Parity_\scS(X,\E)
\]
are isomorphic.  The same reasoning proves~\eqref{eqn:ps-bc-inc} when $h$ is a closed inclusion.  Both~\eqref{eqn:ps-pull} and~\eqref{eqn:ps-bc} hold similarly when $h$ is an open inclusion.

The remaining cases can now be handled by adjunction.  Observe that $f_* \circ h'_{(*)}$ is right adjoint to $h^{\prime(*)} \circ f^*$, and $h_{(*)} \circ f'_*$ is right adjoint to $f^{\prime*} \circ h^{(*)}$.  Thus, when $h$ is an open inclusion, the first isomorphism in~\eqref{eqn:ps-push} follows from a known case of~\eqref{eqn:ps-pull}.  Similar reasoning establishes the second isomorphism in~\eqref{eqn:ps-push} for open inclusions as well.  Likewise, we deduce~\eqref{eqn:ps-bc-inc} from~\eqref{eqn:ps-bc} when $h$ is an open inclusion.  If $h$ a closed inclusion, then we deduce~\eqref{eqn:ps-pull} and~\eqref{eqn:ps-bc} from~\eqref{eqn:ps-push} and~\eqref{eqn:ps-bc-inc}, respectively.
\end{proof}

\subsection{Modular reduction and $\Hom$ spaces}
\label{ss:Hom}

If $\cF^\bullet, \cG^\bullet$ are bounded complexes of objects of $\Parity_\scS(X,\E)$, one can form, in the usual way, a complex of $\E$-modules $\uHom^\bullet(\cF^\bullet,\cG^\bullet)$ whose $k$-th term is given by $\prod_{j-i = k} \Hom(\cF^i, \cG^j)$.  We denote the image of this complex in $\Db(\E\lh\ngmod)$ by $\RHom(\cF^\bullet, \cG^\bullet)$.  This construction gives us a triangulated bifunctor
\[
\RHom: \Dmix_\scS(X,\E)^\op \times \Dmix_\scS(X,\E) \to \Db(\E\lh\ngmod)
\]
satisfying $\Hom(\cF^\bullet,\cG^\bullet) \cong H^0(\RHom(\cF^\bullet,\cG^\bullet))$.  

\begin{lem}
\label{lem:rhom-scalars}
For $\cF, \cG \in \Dmix_\scS(X,\O)$, there exists a natural bifunctorial isomorphism
\[
\F \lotimes_\O \RHom(\cF,\cG) \simto \RHom \bigl( \F(\cF), \F(\cG) \bigr).
\]
\end{lem}

\begin{proof}
For $\cF^\bullet$, $\cG^\bullet$ bounded complexes of objects of $\Parity_{\scS}(X,\O)$, we have a natural and bifunctorial morphism $\uHom^\bullet(\cF^\bullet,\cG^\bullet) \to \uHom^\bullet \bigl( \F(\cF^\bullet),\F(\cG^\bullet) \bigr)$. Passing to derived (resp.~homotopy) categories and using adjunction we deduce a morphism of bifunctors
\[
\F \lotimes_\O \RHom(-,-) \to \RHom \bigl( \F(-), \F(-) \bigr).
\]
The fact that this morphism is an isomorphism follows from the fact that for $\cF,\cG$ in $\Parity_{\scS}(X,\O)$ the $\O$-module $\Hom(\cF,\cG)$ is free, and the natural morphism $\F \otimes_\O \Hom(\cF,\cG) \to \Hom(\F(\cF), \F(\cG))$ is an isomorphism. (See the proof of~\cite[Proposition~2.39]{jmw}.)
\end{proof}

\begin{lem}\label{lem:o-f-ses}
For $\cF, \cG \in \Dmix_\scS(X,\O)$, there is a natural short exact sequence
\[
0 \to \F \otimes_\O \Hom(\cF,\cG) \to \Hom \bigl( \F(\cF), \F(\cG) \bigr) \to \Tor^\O_1 \bigl( \F,\Hom(\cF,\cG[1]) \bigr) \to 0.
\]
\end{lem}
\begin{proof}
In $\Db(\O\lh\ngmod)$, consider the distinguished triangle
\[
\tau_{\le 0}\RHom(\cF,\cG) \to \RHom(\cF,\cG) \to \tau_{\ge 1}\RHom(\cF,\cG) \to,
\]
where $\tau_{\le 0}$ and $\tau_{\ge 1}$ are the truncation functors with respect to the natural t-structure.  Now apply the functor $\F \lotimes_\O ({-})$.  Since $\O\lh\ngmod$ has global dimension~$1$, the object $\F \lotimes_\O \tau_{\ge 1}\RHom(\cF,\cG)$ can have nonzero cohomology only in degrees${}\ge 0$.  On the other hand, since $\F \lotimes_\O ({-})$ is right t-exact, $\F \lotimes_\O \tau_{\le 0}\RHom(\cF,\cG)$ can have nonzero cohomology only in degrees${}\le 0$.  Thus, as a portion of the long exact cohomology sequence, we find the short exact sequence
\begin{multline*}
0 \to H^0 \bigl( \F \lotimes_\O \tau_{\le 0}\RHom(\cF,\cG) \bigr) \to H^0 \bigl( \F \lotimes_\O \RHom(\cF,\cG) \bigr) \to \\
H^0 \bigl( \F \lotimes_\O \tau_{\ge 1}\RHom(\cF,\cG) \bigr) \to 0.
\end{multline*}
By right t-exactness again, the first term is identified with $\F \otimes_\O H^0 \bigl( \RHom(\cF,\cG) \bigr) \cong \F \otimes_\O \Hom(\cF,\cG)$.  By Lemma~\ref{lem:rhom-scalars}, the second is identified with $\Hom \bigl( \F(\cF),\F(\cG) \bigr)$.  For the last term, we study the distinguished triangle
\[
H^1 \bigl( \RHom(\cF,\cG) \bigr) [-1] \to \tau_{\ge 1}\RHom(\cF,\cG) \to \tau_{\ge 2}\RHom(\cF,\cG) \to.
\]
We apply $\F\lotimes_\O ({-})$ again.  Since $H^i \bigl( \F \lotimes_\O \tau_{\ge 2}\RHom(\cF,\cG) \bigr) = 0$ for $i \in \{-1,0\}$, we have $H^0 \bigl( \F \lotimes_\O \tau_{\ge 1}\RHom(\cF,\cG) \bigr) \cong H^0 \bigl( \F \lotimes_\O \Hom(\cF,\cG[1])[-1] \bigr) \cong H^{-1} \bigl( \F \lotimes_\O \Hom(\cF,\cG[1]) \bigr) \cong \Tor^\O_1 \bigl( \F,\Hom(\cF,\cG[1]) \bigr)$, as desired.
\end{proof}

\section{Mixed perverse sheaves}
\label{sec:mixed-perv}

The results of \S\S\ref{ss:adjoints}--\ref{ss:loc-closed} tell us that in the setting of $\Dmix_\scS(X,\E)$, we have available the full complement of $*$- and $!$-type pullback and push-forward functors for all locally closed inclusions of unions of strata, satisfying the usual adjunction and composition properties.  We will henceforth follow the usual sheaf-theoretic conventions for denoting these functors, dropping the extra parentheses that were used in the previous section.

\subsection{Perverse t-structure}
\label{ss:perverse-t-structure}

In this subsection we define the perverse t-structure on the category $\Dmix_\scS(X,\E)$.  For any $s \in \scS$ we define the objects
\[
\dmix_s := i_{s!} \uuE_{X_s}, \qquad \nmix_s := i_{s*} \uuE_{X_s}
\]
in $\Dmix_\scS(X,\scS)$. We will informally refer to these objects as \emph{standard} and \emph{costandard} sheaves, respectively.  (This terminology will be justified in~\S\ref{ss:qh-structure}.)  By \eqref{eqn:D!*-2} we have
\begin{equation}\label{eqn:dual-standard}
\D_X(\dmix_s)=\nmix_s.
\end{equation}
Moreover, because $i_{s!}$ and $i_{s*}$ commute with $\K(-)$ and $\F(-)$, see~\eqref{eqn:inclusions-scalars}, we have
\begin{equation}\label{eqn:standard-scalars}
\begin{aligned}
\K(\dmix_s(\O)) \cong \dmix_s(\K), & \quad \K(\nmix_s(\O)) \cong \nmix_s(\K), \\ 
\F(\dmix_s(\O)) \cong \dmix_s(\F), & \quad \F(\nmix_s(\O)) \cong \nmix_s(\F).
\end{aligned}
\end{equation}

\begin{lem}\label{lem:t-structure}
Suppose $X = X_s$ consists of a single stratum.  There is an equivalence of categories
\[
H: \Dmix_\scS(X,\E) \simto \Db(\E\lh\gmod)
\]
such that $H(\uuE{}_X) \cong \E$, and such that $H$ commutes with $\la 1\ra$.
\end{lem}
In the statement and proof of this lemma, we denote by $\la 1\ra: \E\lh\gmod \to \E\lh\gmod$ the shift-of-grading functor defined as follows: for a graded $\E$-module $M = \bigoplus_{n \in \Z} M_n$, we put $(M\la 1\ra)_n = M_{n+1}$.
(This agrees with the convention of~\cite[\S1.8]{ar}, but is opposite to that of~\cite{rsw}.)

\begin{proof}
Let $\Projf^\Z(\E)$ denote the additive category of graded finitely generated projective (or equivalently free) $\E$-modules.  Consider the equivalence of categories $\gamma: \Parity_\scS(X,\E) \simto \Projf^\Z(\E)$ defined by
\[
\gamma(\cF) = \bigoplus_j \mathbb{H}^{-j-\dim X}(X,\cF) \la -j \ra.
\]
(Here, each cohomology group is regarded as a graded $\E$-module concentrated in degree $0$.)
This equivalence satisfies $\gamma \circ \{1\} \cong \la -1\ra \circ \gamma$ and $\gamma(\uuE{}_{X})=\E$. It induces an equivalence $\gamma: \Dmix_\scS(X,\E) \simto \Kb\Projf^\Z(\E)$.

There is an obvious equivalence $\Kb\Projf^\Z(\E) \simto \Db(\E\lh\gmod)$, but the composition $\Dmix_\scS(X,\E) \overset{\gamma}{\to} \Kb\Projf^\Z(\E) \simto \Db(\E\lh\gmod)$ does \emph{not} commute with $\la 1\ra$.  To achieve the latter property, we will construct an autoequivalence of $\Kb\Projf^\Z(\E)$ that is similar in spirit to the ``re-grading functor'' of~\cite[\S9.6]{abg}.

Any object $M \in \Projf^\Z(\E)$ is (in a canonical way) a finite direct sum $M \cong \bigoplus_j M_j \la -j\ra$, where $M_j$ is a free $\E$-module concentrated in degree~$0$.  Given a complex $M^\bullet \in \Kb\Projf^\Z(\E)$, decompose each term $M^i$ in this way: $M^i \cong \bigoplus_j M^i_j \la -j\ra$.  Let us define a new complex $\rho(M^\bullet) \in \Kb\Projf^\Z(\E)$ by
\[
\rho(M^\bullet)^i = \bigoplus_j M^{i+j}_j \la -j\ra,
\]
with the differential induced by that of $M^\bullet$.  It is clear that $\rho: \Kb\Projf^\Z(\E) \to \Kb\Projf^\Z(\E)$ is an equivalence of triangulated categories, and that it satisfies $\rho \circ \la 1\ra \cong [-1]\la 1\ra \circ \rho$ and $\rho(\E)=\E$.  Finally, let $H$ be the composition
\[
\Dmix_\scS(X,\E) \xrightarrow{\gamma} \Kb\Projf^\Z(\E) \xrightarrow{\rho} \Kb\Projf^\Z(\E) \simto \Db(\E\lh\gmod).
\]
One can easily check that $H(\uuE)\cong\E$ and $H \circ \la 1\ra \cong \la 1\ra \circ H$, as desired.
\end{proof}

The following observation will be useful in the sequel.

\begin{lem}
\label{lem:vanishing-d-m}
For $s,t \in \scS$, we have
\[
\Hom_{\Dmix_\scS(X,\E)}(\dmix_s,\nmix_t \langle n \rangle [i])\cong
\begin{cases}
\E & \text{if $s = t$ and $n = i = 0$,} \\
0 & \text{otherwise.}
\end{cases}
\]
\end{lem}
\begin{proof}
Assume first that $s \neq t$. Replacing if necessary $X$ by $\overline{X_s} \cup \overline{X_t}$, one can assume that either $X_s$ or $X_t$ is open in $X$. In the former case we have
\[
\Hom_{\Dmix_\scS(X,\E)}(\dmix_s,\nmix_t \langle n \rangle [i]) \cong \Hom_{\Dmix_\scS(X_s,\E)}(\uuE{}_{X_s}, i_s^* \nmix_t \langle n \rangle [i])=0
\]
since $\nmix_t$ is supported on $\overline{X_t}$. In the latter case we have
\[
\Hom_{\Dmix_\scS(X,\E)}(\dmix_s,\nmix_t \langle n \rangle [i]) \cong \Hom_{\Dmix_\scS(X_t,\E)}(i_t^* \dmix_s, \uuE{}_{X_t} \langle n \rangle [i])=0
\]
since $\dmix_s$ is supported on $\overline{X_s}$.

Now assume $s = t$. Replacing if necessary $X$ by $\overline{X_s}$, one can assume that $X_s$ is open in $X$. Then adjunction and Lemma~\ref{lem:t-structure} give 
\[
\Hom(\dmix_s,\nmix_t \langle n \rangle [i]) \cong \Ext^i_{\E\lh\gmod}(\E, \E \la n\ra),
\]
and the result follows.
\end{proof}

\begin{defn}\label{defn:t-structure}
If $X$ consists of a single stratum, the \emph{perverse t-structure} on $\Dmix_\scS(X,\E)$, denoted by $(\p\Dmix_\scS(X,\E)^{\leq 0},\p\Dmix_\scS(X,\E)^{\ge 0})$, is the transport of the natural t-structure on $\Db(\E\lh\gmod)$ via the equivalence of Lemma~\ref{lem:t-structure}.

If $X$ consists of more than one stratum, the \emph{perverse t-structure} on  $\Dmix_\scS(X,\E)$ is the t-structure given by
\begin{align*}
\p \Dmix_\scS(X,\E)^{\le 0} &= \{ \cF \in \Dmix_\scS(X,\E) \mid \text{for all $s \in \scS$, $i_s^*\cF \in \p\Dmix_\scS(X_s,\E)^{\le 0}$} \}, \\
\p \Dmix_\scS(X,\E)^{\ge 0} &= \{ \cF \in \Dmix_\scS(X,\E) \mid \text{for all $s \in \scS$, $i_s^!\cF \in \p\Dmix_\scS(X_s,\E)^{\ge 0}$} \}.
\end{align*}
(The fact that the categories above constitute a t-structure follows from the general theory of recollement~\cite[\S 1.4]{bbd}.)
The heart of this t-structure is denoted
\[
\Perv^\mix_\scS(X,\E),
\]
and objects in the heart are called \emph{mixed perverse sheaves}.
\end{defn}

The perverse t-structure is clearly bounded. Note also that the Tate twist $\la 1 \ra$ is t-exact for the perverse t-structure.  

For the next statement, choose a uniformizer $\varpi \in \O$, and let ${}' \nmix_s$ be the cone of the morphism $\varpi \cdot \id: \nabla_s(\O) \to \nabla_s(\O)$.

\begin{prop}\label{prop:t-structure}
The perverse t-structure on $\Dmix_\scS(X,\E)$ is uniquely characterized by each of the following statements:
\begin{enumerate}
\item $\p\Dmix_\scS(X,\E)^{\le 0}$ is generated under extensions by the $\dmix_s\la n\ra[m]$ with $s \in \scS$, $n \in \Z$, and $m \ge 0$.
\item If $\E = \O$, $\p\Dmix_\scS(X,\O)^{\ge 0}$ is generated under extensions by the $\nmix_s\la n\ra[m]$ and ${}'\nmix_s\la n\ra[m]$ with $s \in \scS$, $n \in \Z$, and $m \le 0$.

If $\E = \K$ or $\F$, $\p\Dmix_\scS(X,\E)^{\ge 0}$ is generated under extensions by the $\nmix_s\la n\ra[m]$ with $s \in \scS$, $n \in \Z$, and $m \le 0$.
\end{enumerate}
Moreover, when $\E = \K$ or $\F$, the functor $\D_X$ is t-exact.
\end{prop}
\begin{proof}
Let $\sD^{\le 0} \subset \Dmix_\scS(X,\E)$ be the smallest full subcategory that is stable under extensions and contains all $\dmix_s\la n\ra[m]$.  Likewise, let $\sD^{\ge 0}$ be the smallest full subcategory that is stable under extensions and contains all $\nmix_s\la n\ra[m]$ and, if $\E = \O$, all ${}'\nmix_s\la n\ra[m]$ as well.  It is easy to see that $\sD^{\le 0} \subset \p\Dmix_\scS(X,\E)^{\le 0}$, and that $\sD^{\ge 0} \subset \p\Dmix_\scS(X,\E)^{\ge 0}$.  To see that these containments are equalities, it suffices to show that $(\sD^{\le 0}, \sD^{\ge 0})$ constitutes a t-structure.  This can be done by induction on the number of strata, copying the proof of~\cite[Proposition~1]{bezru}.

In the case where $\E = \K$ or $\F$, the stability of the t-structure under $\D_X$ follows from~\eqref{eqn:dual-standard}.
\end{proof}

\begin{rmk}
When $\E = \O$, the functor $\D_X$ is \emph{not} t-exact, even when $X$ is a single stratum.  Instead, the category
\[
\Perv^{+,\mix}_\scS(X,\O) := \D_X(\Perv^\mix_\scS(X,\O))
\]
is the heart of a different t-structure, called the \emph{$p^+$-perverse t-structure}.  An account of $p^+$-perverse sheaves in the unmixed setting can be found in~\cite[\S2.6]{juteau}; the main properties hold in the mixed setting as well, with the same proofs.  Mixed $p^+$-perverse sheaves will not be used in this paper.
\end{rmk}

The truncation functors for the perverse t-structure will be denoted by
\[
\ptau_{\le 0}: \Dmix_\scS(X,\E) \to \p\Dmix_\scS(X,\E)^{\le 0}, \qquad
\ptau_{\ge 0}: \Dmix_\scS(X,\E) \to \p\Dmix_\scS(X,\E)^{\ge 0},
\]
and the cohomology functors by
\[
\pH^i: \Dmix_\scS(X,\E) \to \Perv^\mix_\scS(X,\E).
\]
For each $s \in \scS$, we define an object $\IC^\mix_s(\E) \in \Perv^\mix_\scS(X,\E)$ by
\[
\IC^\mix_s(\E) := \im (\ptau_{\ge 0}\dmix_s(\E) \to \ptau_{\le 0}\nmix_s(\E)).
\]
If $\E=\K$ or $\F$, each $\IC^\mix_s(\E)$ is self-Verdier-dual, and the objects $\{\IC_s^\mix(\E) \la n \ra \mid s \in \scS, n \in \Z\}$ form a complete set of pairwise non-isomorphic simple objects of the finite-length abelian category $\Perv^\mix_{\scS}(X,\E)$.  See \cite[Proposition 1.4.26]{bbd}.

It follows from Proposition~\ref{prop:t-structure} and \eqref{eqn:standard-scalars} that $\K(-) : \Dmix_\scS(X,\O) \to \Dmix_\scS(X,\K)$ is t-exact. In particular, it follows that we have
\[
\K(\IC^\mix_s(\O)) = \IC_s^\mix(\K).
\]
By similar arguments,  the functor $\F(-) : \Dmix_\scS(X,\O) \to \Dmix_\scS(X,\F)$ is right t-exact. More precisely, if $\cF$ is in $\Perv^\mix_{\scS}(X,\O)$ then $\pH^i(\F(\cF))=0$ for $i \notin \{0,-1\}$.  A study of the stalks and costalks of $\IC^\mix_s(\O)$ in this spirit (using \cite[Corollaire 1.4.24]{bbd}) shows that
\begin{equation}
\label{eqn:FIC}
\F(\IC^\mix_s(\O)) \in \Perv^\mix_\scS(X,\F).
\end{equation}

The following lemma follows from the definitions, and the fact that $i_s^*$, $i_s^!$ and the equivalence of Lemma~\ref{lem:t-structure} commute with the functors $\F(-)$.

\begin{lem}
\label{lem:perv-F}
If $\cF \in \Dmix_\scS(X,\O)$ is such that $\F(\cF) \in \Perv^\mix_\scS(X,\F)$, then $\cF \in \Perv^\mix_\scS(X,\O)$.
\end{lem}

In the next lemma, we set $\IC_s(\E):=i_{s!*} \uuE_{X_s}$ (where $i_{s!*}$ is the the usual, non-mixed, intermediate extension functor).

\begin{lem}\label{lem:ic-smooth}
If $\cE_s(\E) \cong \IC_s(\E)$, then $\IC^\mix_s(\E) \cong \IC_s(\E)$ as well.  In particular, if $\overline{X_s}$ is smooth, then $\IC^\mix_s(\E) \cong \uuE_{\overline{X_s}}$.
\end{lem}

\begin{proof}
By assumption $\IC_s(\E)$ is a parity complex, so that it
makes sense as an object of $\Dmix_\scS(X,\E)$.  The result then follows from the criterion in~\cite[Corollaire~1.4.24]{bbd}, using Remark~\ref{rmk:restriction-X_s}.
\end{proof}

The next few lemmas deal with the setting of~\S\ref{ss:stratified}.

\begin{lem}\label{lem:ps-delta-nabla}
Let $f: X \to Y$ be a proper, smooth stratified morphism of relative dimension $d$.  If $X_s \subset f^{-1}(Y_t)$, then
\[
f_*\dmix_s \cong \dmix_t\{\dim Y_t - \dim X_s\}
\quad\text{and}\quad
f_*\nmix_s \cong \nmix_t\{\dim X_s - \dim Y_t\}.
\]
Furthermore, when $\E = \O$, we have $f_*({}'\nmix_s) \cong {}'\nmix_t\{ \dim X_s - \dim Y_t\}$.
\end{lem}

\begin{proof}
It is clear that the morphism $f_{s,t} : X_s \to Y_t$ induces functors $f_{st!}$ and $f_{st}^!$ between $\Parity_{\scS}(X_s,\E)$ and $\Parity_{\scT}(Y_t,\E)$, and hence also functors denoted similarly between $\Dmix_{\scS}(X_s,\E)$ and $\Dmix_{\scT}(Y_t,\E)$. And it follows from the definitions and Remark~\ref{rmk:restriction-X_s} that we have an isomorphism
$i_s^! f^! \cong f_{st}^! i_t^!$.
By adjunction we deduce an isomorphism $f_! i_{s!} \cong i_{t!} f_{st!}$. Applying this isomorphism to $\uuE_{X_s}$ we deduce the first isomorphism from the observation that $f_{st!} \uuE_{X_s} \cong \uuE_{Y_t} \{\dim Y_t - \dim X_s\}$.

The second isomorphism can be proved similarly. Finally, when $\E = \O$, the last assertion follows from the fact that $f_*$ takes $\varpi\cdot\id: \nmix_s \to \nmix_s$ 
to $\varpi\cdot\id: \nmix_t\{\dim X_s - \dim Y_t\} \to \nmix_t\{\dim X_s - \dim Y_t\}$.
\end{proof}

\begin{cor}\label{cor:ps-exactness}
Let $f: X \to Y$ be a proper, smooth stratified morphism.
\begin{enumerate}
\item The functor $f_\dag$ is right t-exact.\label{it:exact_dag}
\item The functor $f_\ddag$ is left t-exact.\label{it:exact_ddag}
\item The functor $f^\dag$ is t-exact.\label{it:exact^dag}
\end{enumerate}
\end{cor}
\begin{proof}
Parts~\eqref{it:exact_dag} and~\eqref{it:exact_ddag} follow from Lemma~\ref{lem:ps-delta-nabla} and the description of the t-structure in Proposition~\ref{prop:t-structure}. For part~\eqref{it:exact^dag}, since $f^\dag$ has a right t-exact left adjoint, it is left t-exact, and since it has a left t-exact right adjoint, it is right t-exact. 
\end{proof}

\begin{cor}\label{cor:ps-IC}
Let $f: X \to Y$ be a proper, smooth stratified morphism with connected fibers.  For any stratum $Y_t \subset Y$, we have
\[
f^\dag\IC^\mix_t \cong \IC^\mix_s,
\]
where $X_s$ is the unique open stratum in $f^{-1}(Y_t)$.
\end{cor}
Note that $f^{-1}(Y_t)$ does indeed contain a unique open stratum: by the assumptions on $f$, it is smooth and connected, and hence irreducible.
\begin{proof}
Let $i'_s: X_s \hookrightarrow f^{-1}(Y_t)$ and $h: f^{-1}(Y_t) \hookrightarrow X$ be the inclusion maps, so that $i_s = h \circ i'_s$.  Using Proposition~\ref{prop:proper-smooth} and the t-exactness of $f^\dag$ (see Corollary~\ref{cor:ps-exactness}), we see that $f^\dag\IC^\mix_t \cong h_{!*}\uuE_{f^{-1}(Y_t)}$.  (Here, as usual, $h_{!*}$ denotes the image of the natural map $\p\tau_{\ge 0}h_! \to \p\tau_{\le 0}h_*$.)  By Lemma~\ref{lem:ic-smooth}, $\uuE_{f^{-1}(Y_t)} \cong (i'_s)_{!*}\uuE_{X_s}$, so $f^\dag\IC^\mix_t \cong (i_s)_{!*}\uuE_{X_s} \cong \IC^\mix_s$, as desired.
\end{proof}

\subsection{Quasihereditary structure for field coefficients}
\label{ss:qh-structure}

We now impose an additional hypothesis on our space $X$:
\begin{enumerate}
\item[\bf(A2)] For each $s \in \scS$, the objects $\dmix_s(\E)$ and $\nmix_s(\E)$ are perverse.
\end{enumerate}
This hypothesis, an analogue of~\cite[Corollaire~4.1.3]{bbd}, will remain in effect for the remainder of Section~\ref{sec:mixed-perv}.  
In Section~\ref{sec:kac-moody} we will prove that (partial) flag varieties of Kac--Moody groups (endowed with the Bruhat stratification) satisfy this condition.
Since $\F(\dmix_s(\O)) \cong \dmix_s(\F)$, Lemma~\ref{lem:perv-F} tells us that it is enough to check~{\bf (A2)} when $\E = \K$ or $\F$.

Note that under hypothesis~{\bf (A2)}, the objects ${}'\nmix_s$ appearing in Proposition~\ref{prop:t-structure} are perverse.  Indeed, that proposition tells us that ${}'\nmix_s$ lies in $\p\Dmix_\scS(X,\O)^{\ge 0}$.  But because it is the cone of a morphism $\nmix_s(\O) \to \nmix_s(\O)$ in $\Perv^\mix_\scS(X,\O)$, it also lies in $\p\Dmix_\scS(X,\O)^{\le 0}$.

\begin{prop}\label{prop:qh-structure}
Assume that $\E = \K$ or $\F$.  Then $\Perv^\mix_\scS(X,\E)$ is a graded quasihereditary category in the sense of Definition~{\rm \ref{defn:qhered}}, and the $\dmix_s\la n\ra$ (resp.~$\nmix_s\la n\ra$) are precisely the standard (resp.~costandard) objects therein.
\end{prop}
\begin{proof}
The only axiom in Definition~\ref{defn:qhered} which may not be an obvious consequence of the assumption~{\bf (A2)} and the theory of recollement is the last one, which requires that
\[
\Ext^2_{\Perv^\mix_\scS(X,\E)}(\dmix_s,\nmix_t \langle n \rangle)=0
\]
for all $s,t \in \scS$ and $n \in \Z$.
To see this, recall (see e.g.~\cite[Lemma 3.2.3]{bgs}) that there is a natural injective morphism
\[
\Ext^2_{\Perv^\mix_\scS(X,\E)}(\dmix_s,\nmix_t \la n \ra) \hookrightarrow \Hom_{\Dmix_\scS(X,\E)}(\dmix_s,\nmix_t \langle n \rangle[2]).
\]
The right-hand side is zero by Lemma~\ref{lem:vanishing-d-m}, and hence so is the left-hand side.
\end{proof}

We can now invoke Proposition~\ref{prop:tilt-class}: if $\E = \K$ or $\F$, then for each $s \in \scS$, there is a unique indecomposable tilting object in $\Perv^\mix_\scS(X,\E)$ that is supported on $\overline{X_s}$ and whose restriction to $X_s$ is $\uuE_{X_s}$.  We denote this object by
$\cT^\mix_s(\E)$.
Tilting objects will be further discussed in~\S\ref{ss:equivalences}.
Similarly, we have indecomposable projective objects $\cP^\mix_s(\E)$ for $s \in \scS$.

\subsection{Projective and tilting objects for general coefficients}
\label{ss:qh-structure-O}

The general machinery of Appendix~\ref{sec:homological} certainly cannot apply to $\Perv^\mix_\scS(X,\O)$, since the latter is not linear over a field.  Nevertheless, we continue the practice from~\S\ref{ss:perverse-t-structure} of referring to the objects $\dmix_s(\O)\la  n\ra$ (resp.~$\nmix_s(\O)\la n\ra$) as ``standard'' (resp.~``costandard'') objects.  We will also freely use the terminology of Definition~\ref{defn:tilting} concerning standard and costandard filtrations.

We will see below that the structure theory of projective and tilting objects in $\Perv^\mix_\scS(X,\O)$ strongly resembles that of quasihereditary categories.  For ordinary (not mixed) perverse $\O$-sheaves, the theory of projective perverse $\O$-sheaves is developed in detail in~\cite[\S\S 2.3--2.4]{rsw}, and that of tilting perverse $\O$-sheaves is developed in~\cite[\S\S B.2--B.3]{ar}. The theory developed in Section~\ref{sec:mixed-der} and the present section allows us to copy the arguments of those sources almost verbatim.  (The only modifications are those needed to keep track of Tate twists.)  We restate those results below without proof.

\begin{lem}[{\cite[Lemma~2.1.6]{rsw}}]\label{lem:krull-schmidt-O}
The category $\Perv^\mix_\scS(X,\O)$ satisfies the Krull--Schmidt property.
\end{lem}

\begin{prop}[{\cite[\S\S2.3--2.4]{rsw}}]\label{prop:projective-O}
\begin{enumerate}
\item If $\cP$ is a projective object in the category $\Perv^\mix_\scS(X,\O)$, then $\F(\cP)$ is a projective object in $\Perv^\mix_\scS(X,\F)$.
\item The category $\Perv^\mix_\scS(X,\O)$ has enough projectives.  Every projective admits a standard filtration, and every object admits a finite projective resolution.
\item If $\cP, \cP' \in \Perv^\mix_\scS(X,\O)$ are both projective, then $\Hom(\cP,\cP')$ is a free $\O$-module, and the natural map $\F \otimes_\O \Hom(\cP,\cP') \to \Hom(\F(\cP), \F(\cP'))$ is an isomorphism.
\item For any $s \in \scS$, there exists a unique indecomposable projective object $\cP^\mix_s(\O)$ in $\Perv^\mix_\scS(X,\O)$ such that $\F(\cP^\mix_s(\O)) \cong \cP^\mix_s(\F)$.  Any projective object in $\Perv^\mix_\scS(X,\O)$ is isomorphic to a direct sum of various $\cP^\mix_s(\O)\la n\ra$.
\end{enumerate}
\end{prop}

\begin{prop}[{\cite[Proposition~B.3]{ar}}]\label{prop:tilting-O}
\begin{enumerate}
\item If $\cT \in \Perv^\mix_\scS(X,\O)$ is tilting, then $\F(\cT)$ is a tilting object in $\Perv^\mix_\scS(X,\F)$.
\item If $\cT, \cT' \in \Perv^\mix_\scS(X,\O)$ are both tilting, then $\Hom(\cT,\cT')$ is a free $\O$-module, and the natural map $\F \otimes_\O \Hom(\cT,\cT') \to \Hom(\F(\cT), \F(\cT'))$ is an isomorphism.
\item For any $s \in \scS$, there exists a unique indecomposable tilting object $\cT^\mix_s(\O)$ in $\Perv^\mix_\scS(X,\O)$ such that $\F(\cT^\mix_s(\O)) \cong \cT^\mix_s(\F)$.  Any tilting object in $\Perv^\mix_\scS(X,\O)$ is isomorphic to a direct sum of various $\cT^\mix_s(\O)\la n\ra$.
\end{enumerate}
\end{prop}

\subsection{Tilting objects and equivalences}
\label{ss:equivalences}

In this subsection, $\E$ may be any of $\K$, $\O$, or $\F$.  The objects $\cT^\mix_s(\E)$ have now been defined in all cases.  We denote by
\[
\Tilt^\mix_\scS(X,\E) \subset \Perv^\mix_\scS(X,\E)
\]
the full additive subcategory consisting of tilting objects.  

\begin{lem}[{\cite[Lemma~B.5]{ar}}]\label{lem:perv-dereq}
The natural functors
\[
\Kb \Tilt^\mix_{\scS}(X,\E) \to \Db \Perv^\mix_\scS(X,\E) \to \Dmix_\scS(X,\E).
\]
are equivalences of categories.
\end{lem}

\begin{lem}[{\cite[Lemma~B.6]{ar}}]\label{lem:perv-ext-scalars}
The following diagram commutes up to isomorphism of functors:
\[
\xymatrix@C=1.5cm{
\Kb \Tilt^\mix_\scS(X,\O) \ar[r]^-{\sim} \ar[d]_{\Kb \F^0} &
  \Db \Perv^\mix_\scS(X,\O) \ar[r]^-{\sim} \ar[d]_{L \F^0} &
  \Dmix_\scS(X,\O) \ar[d]^{\F({-})} \\
\Kb \Tilt^\mix_\scS(X,\F) \ar[r]^-{\sim} &
  \Db \Perv^\mix_\scS(X,\F) \ar[r]^-{\sim} &
  \Dmix_\scS(X,\F) }
\]
There is a similar commutative diagram for $\K({-})$.
\end{lem}

\begin{lem}\label{lem:tilt-even}
If $n \not\equiv \dim X_s - \dim X_t \pmod 2$, then
\[
(\cT^\mix_s : \dmix_t\la n\ra) = (\cT^\mix_s : \nmix_t\la n\ra) = 0.
\]
\end{lem}
\begin{proof}
Recall the decomposition~\eqref{eqn:dmix-decompose} of $\Dmix_\scS(X,\E)$ into even and odd objects.  Assume without loss of generality that $\dim X_s$ is even.  Then $\cT^\mix_s$ is an even object.  Any standard object occurring in a standard filtration of $\cT^\mix_s$ must also be even, and $\dmix_t\la n\ra$ is even if and only if $n \equiv \dim X_t \pmod 2$.  The same argument establishes the rest of the lemma.
\end{proof}

\subsection{Equivariant mixed perverse sheaves}
\label{ss:equivariant}

In this subsection and the following one we treat some extensions of our theory: we explain how to adapt the constructions to define the mixed \emph{equivariant} derived category, and the mixed derived category of an ind-variety.

Assume as above that we are given a complex algebraic variety $X=\bigsqcup_{s \in \scS} X_s$ equipped with a finite algebraic stratification by affine spaces. Assume in addition that $X$ is endowed with an action of a connected algebraic group $H$, and that each stratum is $H$-stable.   We also assume that, for all $i \in \Z_{\geq 0}$, we have that $\mathbb{H}^{2i+1}_H(\mathrm{pt}; \E)=0$ and $\mathbb{H}^{2i}_H(\mathrm{pt}; \E)$ is a free $\E$-module. This assumption always holds for $\E = \K$. For $\E = \O$ or $\F$, it is equivalent to requiring that the residue characteristic $\ell$ not be a torsion prime for the reductive quotient of $H$; see~\cite[\S 2.6]{jmw}.  By a standard spectral sequence argument, this assumption implies that for all $s \in \scS$ we have isomorphisms of graded algebras $\mathbb{H}^\bullet_H(X_s;\E) \cong \mathbb{H}^\bullet_H(\mathrm{pt};\E)$.

Let $\Db_H(X,\E)$ denote the $H$-equivariant constructible derived category of $X$ in the sense of Bernstein--Lunts~\cite{bl}, and let $\Db_{H,\scS}(X,\E)$ be the full triangulated subcategory of $\Db_H(X,\E)$ consisting of objects that are constructible with respect to our fixed stratification.  

Because $H$ is assumed to be connected, the category of $H$-equivariant local systems on an affine space $X_s$ coincides with the category of (non-equivariant) local systems.  In particular, conditions~\cite[(2.1) and~(2.2)]{jmw} are satisfied, so that we can consider the full additive subcategory
$\Parity_{H,\scS}(X,\E) \subset \Db_{H,\scS}(X,\E)$ consisting of parity complexes. (If the strata of our stratification are precisely the $H$-orbits, then $\Db_{H,\scS}(X,\E) = \Db_H(X,\E)$, and we may simply write $\Parity_H(X,\E)$ instead of $\Parity_{H,\scS}(X,\E)$, and likewise for the other notations introduced below.)

We assume that the following condition is satisfied:
\begin{enumerate}
\item[\bf(A1$'$)] For each $s \in \scS$, there is an indecomposable parity complex $\cE_s'(\E) \in \Db_{H,\scS}(X,\E)$ that is supported on $\overline{X_s}$ and satisfies $i_s^*\cE_s'(\E) \cong \uuE{}_{X_s}$.
\end{enumerate}
We can then study the triangulated category
\[
\Dmix_{H,\scS}(X,\E) := \Kb\Parity_{H,\scS}(X,\E).
\]
The theory developed in \S\S\ref{ss:Dmix}--\ref{ss:stratified} 
goes through for $\Dmix_{H,\scS}(X,\E)$, with essentially the same proofs. In particular we have an obvious analogue of~Proposition~\ref{prop:recollement}, and a Verdier duality functor $\D_X$ which satisfies~\eqref{eqn:D!*-2}.

Note that assumption {\bf (A1$'$)} above implies assumption {\bf (A1)} of~\S\ref{ss:var}, so that we can also consider the category $\Dmix_{\scS}(X,\E)$. The forgetful functor $\Parity_{H,\scS}(X,\E) \to \Parity_{\scS}(X,\E)$ induces a functor $\For : \Dmix_{H,\scS}(X,\E) \to \Dmix_{\scS}(X,\E)$. If $h: Y \hookrightarrow X$ is a locally closed inclusion of a union of strata, one can easily check that the following diagram commutes up to isomorphism:
\begin{equation}\label{eqn:sheaf-forget-commute}
\vcenter{\xymatrix{
\Dmix_{H,\scS}(Y,\E) \ar[r]^-{h_!} \ar[d]_-{\For} & \Dmix_{H,\scS}(X,\E) \ar[d]^-{\For} \\
\Dmix_{\scS}(Y,\E) \ar[r]^-{h_!} & \Dmix_{\scS}(X,\E).
}}
\end{equation}
Similar remarks apply to the functors $h_*$, $h^*$, $h^!$. (This justifies our convention that the functors will be denoted by the same symbol in the equivariant or non-equivariant setting.)  Likewise, if $f: X \to Y$ is an $H$-equivariant proper, smooth, stratified morphism, there are commuative diagrams like that above for $f^\dag$ and $f_*$.

The next task is to equip $\Dmix_{H,\scS}(X,\E)$ with a suitable t-structure.

\begin{lem}\label{lem:eqvt-t-structure}
Suppose $X = X_s$ consists of a single stratum.  There is a unique t-structure on $\Dmix_{H,\scS}(X,\E)$ with respect to which the functors $\la 1 \ra$ and $\For$
are t-exact.  Moreover, $\For$ kills no nonzero object in the heart of that t-structure.
\end{lem}

The argument below is elementary, but somewhat lengthy.  
See Remark~\ref{rmk:solvable} below for remarks on alternative proofs.

\begin{proof}
Let $\Dmix_{H,\scS}(X,\E)^\circ \subset \Dmix_{H,\scS}(X,\E)$ be the full triangulated subcategory generated by the object $\uuE_X$ (i.e., without Tate twists).  We claim that the restriction of $\For$ to $\Dmix_{H,\scS}(X,\E)^\circ$ is fully faithful.  It is enough to check that
\[
\Hom_{\Dmix_{H,\scS}(X,\E)} (\uuE_X, \uuE_X[k]) \xrightarrow{\For}
\Hom_{\Dmix_\scS(X,E)} (\uuE_X, \uuE_X[k])
\]
is an isomorphism for all $k \in \Z$.  When $k \ne 0$, both sides vanish.  When $k = 0$, these $\Hom$-groups can instead be computed in $\Parity_{H,\scS}(X,\E)$ and $\Parity_\scS(X,\E)$, respectively.  Both are free rank-one $\E$-modules generated by $\id: \uuE_X \to \uuE_X$.

Thus, $\For$ lets us identify $\Dmix_{H,\scS}(X,\E)^\circ$ with the full triangulated subcategory of $\Dmix_\scS(X,\E)$ generated by $\uuE_X$.  Composing $\For$ with the equivalence of Lemma~\ref{lem:t-structure}, we find that $\Dmix_{H,\scS}(X,\E)^\circ$ is equivalent to
\[
\left\{ M^\bullet \in \Db(\E\lh\gmod) \,\Big|\, 
\begin{array}{c}
\text{for all $i \in \Z$, the graded $\E$-module}\\
\text{$H^i(M^\bullet)$ is concentrated in degree $0$}
\end{array}
\right\}.
\]
It is easy to see that any complex $M^\bullet \in \Db(\E\lh\gmod)$ satisfying this condition is quasi-isomorphic to a complex whose individual terms are concentrated in degree zero.  In other words, we have an equivalence
\begin{equation}\label{eqn:eqvt-circ-emod}
\Dmix_{H,\scS}(X,\E)^\circ \cong \Db(\E\lh\ngmod),
\end{equation}
where $\E\lh\ngmod$ is identified with the subcategory of $\E\lh\gmod$ consisting of graded modules concentrated in degree zero.

Equip $\Dmix_{H,\scS}(X,\E)^\circ$ with the transport of the natural t-structure on $\Db(\E\lh\ngmod)$.  Let $\cA^\circ$ denote its heart.  We will show that for any $\cF, \cG \in \cA^\circ$, we have
\begin{equation}\label{eqn:eqvt-circ-tate}
\Hom(\cF, \cG\la n\ra[k]) = 0 \qquad\text{if $k \le 1$ and $n \ne 0$.}
\end{equation}
To prove this, we may assume that $\cF$ and $\cG$ are indecomposable.  Note that~\eqref{eqn:eqvt-circ-emod} gives us a classification of the indecomposable objects in $\cA^\circ$.  Suppose first that $\cF \cong \cG \cong \uuE_X$.  (This is the only case to consider if $\E$ is a field.)  The group $\Hom(\uuE_X, \uuE_X\la n\ra[k]) = \Hom(\uuE_X, \uuE_X\{-n\}[n+k])$ is obviously zero if $n+k \ne 0$.  When $n+k = 0$, we have $\Hom(\uuE_X, \uuE_X\{-n\}) \cong \mathbb{H}^{-n}_H(X;\E)$.  Our assumptions on $n$ and $k$ imply that either $-n < 0$ or $-n = 1$.  In both of these cases, we have $\mathbb{H}^{-n}_H(X;\E) = 0$, as desired.

If $\E = \O$, there are other indecomposable objects that arise as the cone of $\varpi^m\cdot\id : \uuO_X \to \uuO_X$, where $\varpi \in \O$ is a uniformizer.  In a minor abuse of notation, we denote the cone of $\varpi^m\cdot \id$ by $\uuline{\O/\varpi^m}{}_X$.  To finish the proof of~\eqref{eqn:eqvt-circ-tate}, there are two additional cases to consider:
\begin{enumerate}
\item $\cF = \uuE_X$ and $\cG = \uuline{\O/\varpi^m}{}_X$.  Apply $\Hom(\uuE_X,{-})$ to the triangle
\begin{equation}\label{eqn:eqvt-circ-tate-les}
\uuE_X\la n\ra[k] \to \cG\la n\ra[k] \to \uuE_X\la n\ra[k+1] \xrightarrow[\varpi^m]{[1]}
\end{equation}
where $n \ne 0$.  If $k \le 0$, it follows from the known cases of~\eqref{eqn:eqvt-circ-tate} that $\Hom(\uuE_X, \cG\la n\ra[k]) = 0$.  For the case $k = 1$, note that the map
\[
\varpi^m: \Hom(\uuE_X, \uuE_X\la n\ra[2]) \to \Hom(\uuE_X, \uuE_X\la n\ra[2])
\]
is injective: if $n \ne -2$, this $\Hom$-group is zero, and if $n = -2$, it is isomorphic to $\mathbb{H}^2_H(X;\E)$, which is assumed to be a free $\E$-module.  Since $\Hom(\uuE_X, \uuE_X\la n\ra[1])$ is already known to vanish, the long exact sequence associated to~\eqref{eqn:eqvt-circ-tate-les} shows that $\Hom(\uuE_X, \cG\la n\ra[1]) = 0$.

\item $\cF = \uuline{\O/\varpi^r}{}_X$ and $\cG = \uuE_X$ or $\uuline{\O/\varpi^m}{}_X$.  Applying $\Hom({-},\cG\la n\ra[k])$ to the distinguished triangle $\uuE_X \to \cF \to \uuE_X[1] \xrightarrow[\varpi^r]{[1]}$, we obtain the sequence
\[
\cdots \to \Hom(\uuE_X, \cG\la n\ra[k-1]) \to \Hom(\cF,\cG\la n\ra[k]) \to \Hom(\uuE_X, \cG\la n\ra[k]) \to \cdots.
\]
The known cases of~\eqref{eqn:eqvt-circ-tate} imply that the middle term vanishes for $k \le 1$ and $n \ne 0$.
\end{enumerate}
We have now proved~\eqref{eqn:eqvt-circ-tate} in all cases.  

Next, we claim that for any $n \in \Z$, we have
\begin{equation}\label{eqn:eqvt-circ-admissible}
\cA^\circ \ast \cA^\circ\la n\ra[1] \subset \cA^\circ\la n\ra[1] \ast \cA^\circ.
\end{equation}
If $n \ne 0$, it follows from~\eqref{eqn:eqvt-circ-tate} (with $k = 0$) that the left-hand side of~\eqref{eqn:eqvt-circ-admissible} contains only objects of the form $\cF \oplus \cG\la n\ra[1]$ with $\cF,\cG \in \cA^\circ$, so the containment asserted in~\eqref{eqn:eqvt-circ-admissible} certainly holds.  On the other hand, if $n = 0$, both sides of~\eqref{eqn:eqvt-circ-admissible} are contained within $\Dmix_{H,\scS}(X,\E)^\circ$.  In this case,~\eqref{eqn:eqvt-circ-admissible} holds because $\cA^\circ$ is the heart of a t-structure (see~\cite[Th\'eor\`eme~1.3.6 and~(1.3.11)(ii)]{bbd}).

Now, let $\cA \subset \Dmix_{H,\scS}(X,\E)$ be the full subcategory whose objects are direct sums of various $\cF\la n\ra$ with $\cF \in \cA^\circ$ and $n \in \Z$.  It is immediate from~\eqref{eqn:eqvt-circ-admissible} that
\begin{equation}\label{eqn:eqvt-admissible}
\cA \ast \cA[1] \subset \cA[1] \ast \cA.
\end{equation}
We claim that for any $\cF, \cG \in \cA$, we have
\begin{equation}\label{eqn:eqvt-neg-ext}
\Hom(\cF,\cG[k]) = 0 \qquad\text{if $k < 0$.}
\end{equation}
We may assume without loss of generality that $\cF \in \cA^\circ$ and $\cG \in \cA^\circ\la n\ra$ for some $n$.  If $n \ne 0$, then~\eqref{eqn:eqvt-neg-ext} follows from~\eqref{eqn:eqvt-circ-tate}.  If $n = 0$, then~\eqref{eqn:eqvt-neg-ext} holds because $\cF$ and $\cG$ both lie in the heart of a t-structure on $\Dmix_{H,\scS}(X,\E)^\circ$.

Finally, we claim that $\cA$ is closed under extensions, i.e., that
\[
\cA \ast \cA \subset \cA.
\]
This follows from the following two observations: (1)~we have $\cA^\circ \ast \cA^\circ \subset \cA^\circ$, because $\cA^\circ$ is the heart of a t-structure; and (2)~for $n \ne 0$, we have $\cA^\circ \ast \cA^\circ\la n\ra = \cA^\circ \oplus \cA^\circ\la n\ra$, as can be seen using~\eqref{eqn:eqvt-circ-tate} for $k=1$.

Because $\cA$ is closed under direct sums and satisfies~\eqref{eqn:eqvt-admissible} and~\eqref{eqn:eqvt-neg-ext}, \cite[Proposition~1.2.4]{bbd} tells us that it is an ``admissible abelian'' category in the sense of~\cite[D\'efinition~1.2.5]{bbd}.  (See also~\cite[(1.3.11)(ii) and Remarque~1.3.14]{bbd}.)  Then, by~\cite[Proposition~1.3.13]{bbd}, since $\cA$ is also closed under extensions, it is the heart of a t-structure on $\Dmix_{H,\scS}(X,\E)$.

It is clear by construction that this new t-structure is the unique t-structure whose heart contains $\cA^\circ$ and is stable under $\la n\ra$.  The latter two properties must be had by any t-structure with respect to which $\la 1 \ra$ and $\For$ are t-exact, so the uniqueness asserted in the lemma holds.  Finally, we see from~\eqref{eqn:eqvt-circ-emod} that $\For$ kills no nonzero object in $\cA^\circ$, and hence no nonzero object of $\cA$.
\end{proof}

\begin{defn}
If $X$ consists of a single stratum, the \emph{perverse t-structure} on $\Dmix_{H,\scS}(X,\E)$, denoted by $(\p\Dmix_{H,\scS}(X,\E)^{\le 0}, \p\Dmix_{H,\scS}(X,\E)^{\ge 0})$, is the t-structure described in Lemma~\ref{lem:eqvt-t-structure}.  If $X$ consists of more than one stratum, the \emph{perverse t-structure} on $\Dmix_{H,\scS}(X,\E)$ is the t-structure given by
\begin{align*}
\p \Dmix_{H,\scS}(X,\E)^{\le 0} &= \{ \cF \in \Dmix_\scS(X,\E) \mid \text{for all $s \in \scS$, $i_s^*\cF \in \p\Dmix_{H,\scS}(X_s,\E)^{\le 0}$} \}, \\
\p \Dmix_{H,\scS}(X,\E)^{\ge 0} &= \{ \cF \in \Dmix_\scS(X,\E) \mid \text{for all $s \in \scS$, $i_s^!\cF \in \p\Dmix_{H,\scS}(X_s,\E)^{\ge 0}$} \}.
\end{align*}
The heart of this t-structure is denoted by $\Perv^\mix_{H,\scS}(X,\E)$, and objects in the heart are called \emph{equivariant mixed perverse sheaves}.
\end{defn}

The following lemma is an easy consequence of Lemma~\ref{lem:eqvt-t-structure} and~\eqref{eqn:sheaf-forget-commute}.

\begin{lem}
The functor $\For: \Dmix_{H,\scS}(X,\E) \to \Dmix_\scS(X,\E)$ is t-exact with respect to the perverse t-structures on both categories.  Moreover, it kills no nonzero object in $\Perv^\mix_{H,\scS}(X,\E)$.  As a consequence, we have
\begin{align*}
\p  \Dmix_{H,\scS}(X,\E)^{\leq 0} & = \{\cF \in \Dmix_{H,\scS}(X,\E) \mid \For(\cF) \in \p \Dmix_{\scS}(X,\E)^{\leq 0}\}; \\
\p  \Dmix_{H,\scS}(X,\E)^{\geq 0} & = \{\cF \in \Dmix_{H,\scS}(X,\E) \mid \For(\cF) \in \p \Dmix_{\scS}(X,\E)^{\geq 0}\}.
\end{align*}
\end{lem}

For any $s \in \scS$ we can define the objects $\dmix_s := i_{s!} \uuE_{X_s}$ and $\nmix_s := i_{s*} \uuE_{X_s}$ in $\Dmix_{H,\scS}(X,\E)$. By~\eqref{eqn:sheaf-forget-commute}, the images of these objects under the functor $\For$ are the objects denoted by the same symbols in $\Dmix_{\scS}(X,\E)$.  All results from~\S\ref{ss:perverse-t-structure} hold in the equivariant setting, including, in particular, analogues of Proposition~\ref{prop:t-structure} and Lemma~\ref{lem:ps-delta-nabla}.

Note that Proposition~\ref{prop:qh-structure} does \emph{not} hold in general in the equivariant setting (even under assumption {\bf (A2)}), because it may happen that 
\[
\Ext^2_{\Perv^\mix_{H,\scS}(X,\E)}(\dmix_s, \nmix_t\la n\ra) \ne 0
\]
for some $s,t \in \scS$ and $n \in \Z$.
For the same reason, the functors in Lemma~\ref{lem:perv-dereq} may fail to be equivalences in the equivariant setting. 

\begin{rmk}\label{rmk:solvable}
When $H$ is a solvable group---as will be the case for our applications in Sections~\ref{sec:kac-moody}--\ref{sec:formality}---the equivariant perverse t-structure on a single stratum admits an alternative description, in terms of Koszul duality.  In this case, the equivariant cohomology ring $\mathbb{H}^\bullet_H(X_s;\E)$ can be identified with the symmetric algebra $\mathrm{S}(V)$ on $V := X^*(H) \otimes_\Z \E$, where $V$ is in degree~$2$.  (Here $X^*(H)$ is the character lattice of $H$.)  A variant of the proof of Lemma~\ref{lem:t-structure} gives us an equivalence of categories
\[
\gamma': \Dmix_{H,\scS}(X_s,\E) \simto \Db(\mathsf{S}\lh\gmod),
\]
where $\mathsf{S}=\mathrm{S}(V)$ with $V$ placed now in degree $-2$, and $\mathsf{S}\lh\gmod$ denotes the category of finitely generated graded $\mathsf{S}$-modules.  Let $\mathsf{T}$ denote the exterior algebra $\bigwedge(V^*)$ on $V^* := \Hom_\E(V,\E)$.  We regard $\mathsf{T}$ as a differential bigraded algebra (or dgg-algebra) by placing $V^*$ in bidegree $(-1,2)$ and equipping it with the trivial differential (of bidegree $(1,0)$).  The version of Koszul duality developed in~\cite{mr,mr2} yields a contravariant equivalence
\[
\mathsf{Koszul}: \Db(\mathsf{S}\lh\gmod) \simto \mathsf{T}\lh\mathsf{dggmod},
\]
where $\mathsf{T}\lh\mathsf{dggmod}$ is the derived category of differential bigraded $\mathsf{T}$-modules whose cohomology is finitely generated over $\mathsf{T}$.  It can be checked that the perverse t-structure on $\Dmix_{H,\scS}(X_s,\E)$ corresponds via $\mathsf{Koszul} \circ \gamma'$ to
the t-structure given by
\[
( \mathsf{T}\lh\mathsf{dggmod} )^{\leq 0} = \{M  \mid \mathsf{H}^{>0}(M)=0\}, \quad
( \mathsf{T}\lh\mathsf{dggmod} )^{\geq 0}  = \{M  \mid \mathsf{H}^{<0}(M)=0\}.
\]
This approach can be extended to a general $H$: under our assumptions $\mathbb{H}^\bullet_H(X_s,\E)$ is a polynomial ring on generators in even positive degrees, and $\mathsf{T}$ should be replaced by an exterior algebra on generators with bidegrees of the form $(-1,2k)$. See also~\cite[\S 11.4]{bl} for a construction of an analoguous t-structure for the ordinary equivariant derived category of a point, with real coefficients, which it might be possible to adapt in the present setting.
\end{rmk}

\subsection{Ind-varieties}
\label{ss:ind-varieties}

The theory of Sections~\ref{sec:mixed-der}--\ref{sec:mixed-perv} can also readily be extended to the setting of ind-varieties.  Suppose that $X$ is an inductive system
\[
X_0 \hookrightarrow X_1 \hookrightarrow X_2 \hookrightarrow \cdots
\]
of complex algebraic varieties.  Assume that each map in this system is a closed inclusion.  Assume also that each $X_n$ is equipped with a finite stratification by affine spaces as in~\S\ref{ss:var}, and that the inclusion $X_n \hookrightarrow X_{n+1}$ identifies each stratum of $X_n$ with a stratum of $X_{n+1}$.  Lastly, assume that condition~{\bf (A1)} holds on each $X_n$.  It then makes sense to speak of the ``strata of $X$,'' and, as explained in~\cite[\S2.7]{jmw}, to work with the category $\Parity_\scS(X,\E)$ of parity complexes on $X$.  Recall that every object of $\Parity_\scS(X,\E)$ is a \emph{finite} direct sum of various $\cE_s(\E)\{m\}$.  In particular, every object of $\Parity_\scS(X,\E)$ has support contained in some  $X_n$.  We define $\Dmix_\scS(X,\E) := \Kb\Parity_\scS(X,\E)$, just as for ordinary varieties.  If $X$ is acted on by a (pro-)algebraic group, one can also consider equivariant versions of these categories, as in~\S\ref{ss:equivariant}.

The proof of Proposition~\ref{prop:recollement} involves an induction argument on the number of strata in $Z$, but no restriction on $U$.  This argument goes through in the ind-variety setting as long as $Z$ is a closed union of finitely many strata.  In particular, we have that the push-forward functor $\Dmix_\scS(Z,\E) \to \Dmix_\scS(X,\E)$ is fully faithful.  Because every object of $\Parity_\scS(X,\E)$ (and hence of $\Dmix_\scS(X,\E)$) is supported on some $X_n$, the natural functor
\[
\varinjlim_n \Dmix_\scS(X_n,\E) \to \Dmix_\scS(X,\E)
\]
is an equivalence of categories. This observation lets us transfer most other results from Sections~\ref{sec:mixed-der}--\ref{sec:mixed-perv} to the ind-variety setting.  In particular, $\Dmix_\scS(X,\E)$ has a perverse t-structure.  Its heart $\Perv^\mix_\scS(X,\E)$ is a noetherian category; if $\E = \K$ or $\F$, it is a finite-length category.  If condition~{\bf (A2)} holds for $X$, then it makes sense to speak of tilting objects in $\Perv^\mix_\scS(X,\E)$; they are classified the same way as in \S\S\ref{ss:qh-structure}--\ref{ss:qh-structure-O}.  Lemma~\ref{lem:perv-dereq} also holds in this setting.  However, $\Perv^\mix_\scS(X,\E)$ does \emph{not}, in general, have enough projectives.

\section{Partial flag varieties of Kac--Moody groups}
\label{sec:kac-moody}

Sections~\ref{sec:kac-moody} and~\ref{sec:formality} are devoted to the study of flag varieties of Kac--Moody groups. According to~\cite[Theorem~4.6]{jmw}, any partial flag variety of a Kac--Moody group (equipped with the Bruhat stratification) satisfies condition~{\bf (A1$'$)}  of \S\ref{ss:equivariant} (where the group ``$H$'' is the Borel subgroup), and hence also condition {\bf (A1)} of~\S\ref{ss:var}.\footnote{In \cite[Theorem~4.6]{jmw}, an assumption is made on the ring of coefficients. However, this assumption is not needed in the case $I=\varnothing$, which is the only case considered in the present paper.}  The main result of this section asserts that they also satisfy condition~{\bf (A2)} of \S\ref{ss:qh-structure}.

\subsection{Notation}
\label{ss:notation}

Let $G$ be a Kac--Moody group (over $\C$), with standard Borel subgroup $B \subset G$ and maximal torus $T \subset B$.  Let $W$ be the Weyl group of $G$, and let $S \subset W$ be the set of simple reflections.  Let  $\cB = G/B$.  The $B$-orbits on $\cB$ are parametrized by $W$; for $w \in W$, we denote the corresponding orbit by $\cB_w$.  Recall that the closures $\overline{\cB_w}$ are (finite-dimensional) projective varieties.  In this way, $\cB$ has the structure of an ind-variety; see~\cite[\S 4.1]{jmw} and the references therein.  Following~\cite{ar,rsw}, we denote by $\Db_{(B)}(\cB,\E)$ the derived category of $\E$-sheaves that are constructible with respect to the stratification by $B$-orbits.

Because $\cB$ satisfies~{\bf (A1)}, the theory developed in Section~\ref{sec:mixed-der} applies to $\cB$.  We will write $\Parity_{(B)}(\cB,\E)$, $\Dmix_{(B)}(\cB,\E)$, etc., for the various categories arising in that theory.  
In \S\S\ref{ss:partial}--\ref{ss:convolution} we will mainly work with the equivariant categories $\Db_B(\cB,\E)$, $\Parity_B(\cB,\E)$, and $\Dmix_B(\cB,\E)$, and with the objects $\dmix_w$, $\nmix_w$ of $\Dmix_B(\cB,\E)$. Of course all the results in these subsections have obvious counterparts in $\Dmix_{(B)}(\cB,\E)$, obtained by applying the forgetful functor $\For :\Dmix_{B}(\cB,\E) \to \Dmix_{(B)}(\cB,\E)$.

\subsection{Projections on partial flag varieties}
\label{ss:partial}

Given a subset $I \subset S$, let $P^I \subset G$ denote the corresponding standard parabolic subgroup.  Its Weyl group, denoted $W_I$, is the subgroup of $W$ generated by $I$.  When $I$ is of finite type (i.e., when $W_I$ is a finite group), the partial flag variety $\scP^I := G/P^I$ is again an ind-variety.  As above, we can consider the categories $\Db_{(B)}(\scP^I,\E)$, $\Parity_{(B)}(\scP^I,\E)$, and $\Dmix_{(B)}(\scP^I,\E)$, and the equivariant counterparts $\Db_{B}(\scP^I,\E)$, $\Parity_{B}(\scP^I,\E)$, and $\Dmix_{B}(\scP^I,\E)$.  Recall that the $B$-orbits on $\scP^I$ are parametrized by the set of left cosets $W/W_I$.  For $w \in W$, we denote by $\overline{w}$ its image in $W/W_I$.  Then one can consider the corresponding $B$-orbit $\scP^I_{\overline{w}} \subset \scP^I$, as well as the objects $\dmix_{\overline{w}}$ and $\nmix_{\overline{w}}$ of $\Db_{B}(\scP^I,\E)$.

Let $\pi^I: \cB \to \scP^I$ denote the natural projection map.  This is a proper, smooth, stratified map; its fibers are isomorphic to the (finite-dimensional) smooth projective variety $P^I/B$.  Let $r_I = \dim P^I/B$, and let $w_I$ denote the longest element of $W_I$.  Then $\ell(w_I) = r_I$.  The variety $\overline{\cB_{w_I}}$ is isomorphic to $P^I/B$; in particular, it is smooth of dimension $r_I$.

Let $W^I \subset W$ be the set of minimal-length representatives for $W/W_I$.  If $w \in W^I$, then of course $ww_I$ is the unique maximal-length representative of $wW_I$.  
In the special case where $I$ is a singleton $\{s\}$, we will write $\pi^s: \cB \to \scP^s$ instead of $\pi^{\{s\}}: \cB \to \scP^{\{s\}}$.  In this case, of course, we have $P^I/B \cong \mathbb{P}^1$ and $r_I = 1$.
If $w \in W$ is such that $ws<w$, from the equivariant analogue of Lemma~\ref{lem:ps-delta-nabla} we deduce that
\begin{equation}\label{eqn:pi-delta-nabla}
\pi^s_\dag \dmix_{w} \cong \dmix_{\overline{w}}
\qquad\text{and}\qquad
\pi^s_\ddag \dmix_{ws} \cong \dmix_{\overline{w}} \{-1\}.
\end{equation}

\begin{lem}
\label{lem:dt-w-ws}
Let $w \in W$, let $s$ be a simple reflection, and assume that $ws<w$. Consider the following morphisms, defined by adjunction:
\begin{align*}
\eta &:\dmix_w \to \pi^{s\dag} \pi^s_\dag \dmix_w \xrightarrow[\sim]{\eqref{eqn:pi-delta-nabla}} \pi^{s\dag} \dmix_{\overline{w}}, \\
\epsilon &:\pi^{s\dag} \dmix_{\overline{w}} \xrightarrow[\sim]{\eqref{eqn:pi-delta-nabla}} \pi^{s\dag} \pi^s_\ddag \dmix_{ws}\{1\} \to \dmix_{ws}\{1\}.
\end{align*}
There exists a morphism $f: \dmix_{ws}\{1\} \to \dmix_w[1]$ that makes the following diagram into a distinguished triangle in $\Dmix_B(\cB,\E)$:
\[
\dmix_w \xrightarrow{\eta} \pi^{s\dag} \dmix_{\overline{w}} \xrightarrow{\epsilon} \dmix_{ws} \{1\} \xrightarrow{f}.
\]
\end{lem}
Of course, there is a similar statement involving costandard objects.
\begin{proof}
Consider the variety $Y^s_w := \cB_w \sqcup \cB_{ws} = (\pi^s)^{-1}(\scP^s_{\overline{w}})$.  This variety is smooth, so the (equivariant) parity sheaf $\cE'_{Y^s_w,w}$ associated with the stratum $\cB_w \subset Y^s_w$ is simply the shifted constant sheaf $\uuE_{Y^s_w} = \underline{\E}_{Y^s_w} \{\ell(w)\}$.  Let $j: \cB_w \hookrightarrow Y^s_w$ and $i: \cB_{ws} \hookrightarrow Y^s_w$ be the inclusion maps.  Applying the functorial distinguished triangle $j_!j^! \to \id \to i_*i^* \to$ to $\uuE_{Y^s_w}$, we obtain a triangle
\[
\dmix_{Y^s_w,w} \to \uuE_{Y^s_w} \to \dmix_{Y^s_w,ws}\{1\} \to.
\]
The middle term can be identified with $\pi^{s\dag}\dmix_{\scP^s_{\overline{w}},\overline{w}}$. (Here we still denote by $\pi^s$ the morphism $Y_w^s \to \scP^s_{\overline{w}}$ induced by $\pi^s$). Taking the $!$-direct image under the embedding $Y^s_w \hookrightarrow \cB$ and using Proposition~\ref{prop:proper-smooth}, we obtain a distinguished triangle 
\begin{equation}
\label{eqn:traingle-pis}
\dmix_w \to \pi^{s\dag} \dmix_{\overline{w}} \to \dmix_{ws} \{1\} \xrightarrow{[1]}.
\end{equation}
Since 
\[
\Hom(\dmix_w,\pi^{s\dag} \dmix_{\overline{w}}) \cong \Hom(\pi^{s}_{\dag} \dmix_w, \dmix_{\overline{w}}) \overset{\eqref{eqn:pi-delta-nabla}}{\cong} \Hom(\dmix_{\overline{w}}, \dmix_{\overline{w}})
\]
is free of rank one over $\E$ and since the first morphism in~\eqref{eqn:traingle-pis} is a generator (since its image under the restriction to $Y^s_w$ is), one can assume that this first morphism is induced by adjunction. By a similar argument one can also assume that the second morphism is induced by adjunction, which finishes the proof.
\end{proof}

\subsection{Convolution}
\label{ss:convolution}

Recall that the (ordinary, not mixed) equivariant derived category $\Db_B(\cB,\E)$ is endowed with a ``convolution product'' 
\[
\star^B: \Db_B(\cB,\E) \times \Db_B(\cB,\E) \to \Db_B(\cB,\E).
\]
If fact, if $m : G \times^B \cB \to \cB$ is the morphism induced by the $G$-action on $\cB$, then, by definition, for $\cF,\cG$ in $\Db_B(\cB,\E)$ we have
\[
\cF \star^B \cG = m_*(\cF \tboxtimes \cG)
\]
where $\cF \tboxtimes \cG$ is the unique object of $\Db_B(G \times^B \cB,\E)$ whose inverse image under the quotient morphism $G \times \cB \to G \times^B \cB$ is $\widetilde{\cF} \boxtimes \cG$; here $\widetilde{\cF}$ is the inverse image of $\cF$ under the projection $G \to \cB$. (To be really precise, this construction is correct only in the case $G$ is of finite type. We leave the necessary modifications in the other cases to the interested reader.) The convolution product is associative.

According to~\cite[Theorem~4.8]{jmw}, the convolution product restricts to a functor $\star^B: \Parity_B(\cB,\E) \times \Parity_B(\cB,\E) \to \Parity_B(\cB,\E)$.  Passing to homotopy categories, we obtain a bifunctor
\[
\star^B : \Dmix_B(\cB,\E) \times \Dmix_B(\cB,\E) \to \Dmix_B(\cB,\E).
\]

\begin{rmk}
The same construction as above gives a functor $\star^B: \Dmix_{(B)}(\cB,\E) \times \Dmix_B(\cB,\E) \to \Dmix_{(B)}(\cB,\E)$.
\end{rmk}

We will now study a number of special cases of convolution products. In the following statement we will consider the objects $\uuE_{\overline{\cB_{w_I}}}$ and $\uuE_{\cB_{e}}$. These objects are (equivariant) parity sheaves, so that they make sense as objects of $\Dmix_B(\cB,\E)$.

\begin{lem}\label{lem:convo-ewi}
Let $I \subset S$ be of finite type.  For any $\cF$ in $\Dmix_B(\cB,\E)$, there is a functorial isomorphism $\theta: \cF \star^B \uuE_{\overline{\cB_{w_I}}} \simto \pi^{I\dag} \pi^I_{\ddag}\cF\{r_I\}$ such that the following diagram commutes:
\[
\xymatrix{
\cF \star^B \uuE_{\overline{\cB_{w_I}}} \ar[r] \ar[d]_{\wr}^\theta & \cF \star^B \uuE_{\cB_{e}} \{r_I\} \ar[d]^{\wr} \\
\pi^{I\dag} \pi^I_{\ddag}\cF\{r_I\} \ar[r] & \cF\{r_I\}.
}
\]
Here, the top horizontal map is induced by the adjunction map $\uuE_{\overline{\cB_{w_I}}} \to i_{e*}i_e^*\uuE_{\overline{\cB_{w_I}}} \cong \uuE_{\cB_{e}}\{r_I\}$, the bottom map by the adjunction $\pi^{I\dag} \pi^I_{\ddag} \to \id$, and the right-hand vertical isomorphism is due to the fact that $\uuE_{\cB_{e}}$ is the unit for the convolution product.
\end{lem}
In the setting of $\ell$-adic \'etale sheaves, this fact is well known; for an explanation, see, e.g., the proof of~\cite[Proposition~12.2]{ar:kdsf}.
\begin{proof}
By the definition of our functors, it suffices to prove the lemma in the case where $\cF$ is a parity complex.  Henceforth, we assume that this is the case.

Consider the space $G \times^B \overline{\cB_{w_I}}$ and the maps $p, m: G \times^B \overline{\cB_{w_I}} \to \cB$ given by $p: (g,hB) \mapsto gB$ and $m: (g,hB) \mapsto ghB$.  By definition, $\cF \star^B \uuE_{\overline{\cB_{w_I}}} \cong m_*(\cF \tboxtimes \uuE_{\overline{\cB_{w_I}}})$.  
In this case, it is easy to see that the twisted external tensor product $\cF \tboxtimes \uuE_{\overline{\cB_{w_I}}}$ is naturally isomorphic to $p^*\cF\{r_I\}$.  We thus have $\cF \star^B \uuE_{\overline{\cB_{w_I}}} \cong m_*p^*\cF \{r_I\}$.  The existence of $\theta$ then follows from the proper base change theorem and the fact the following square is cartesian:
\[
\xymatrix{
G \times^B \overline{\cB_{w_I}} \ar[r]^-p \ar[d]_m & \cB \ar[d]^{\pi^I} \\
\cB \ar[r]_{\pi^I} & \scP^I.
}
\]

Next, let $h: \cB \to G \times^B \overline{\cB_{w_I}}$ be the map $h: gB \mapsto (g,1B)$.  This map is a closed embedding, and $\cF \tboxtimes \uuE_{\cB_{e}} \cong h_*\cF$.  Moreover, the natural map $\cF \tboxtimes \uuE_{\overline{\cB_{w_I}}} \to \cF \tboxtimes \uuE_{\cB_{e}} \{r_I\}$ can be identified with the morphism $p^*\cF \{r_I\} \to h_*h^*p^* \cF \{r_I\} \cong h_*\cF\{r_I\}$ induced by adjunction.  The commutative diagram in the statement of the lemma is obtained from the diagram
\[
\xymatrix{
\cF \tboxtimes \uuE_{\overline{\cB_{w_I}}} \ar[r] \ar[d]_{\wr} & \cF \tboxtimes \uuE_{\cB_{e}}\{r_I\} \ar[d]^{\wr} \\
p^*\cF\{r_I\} \ar[r] & h_*\cF\{r_I\}}
\]
by applying $m_*$.
\end{proof}

\begin{prop}\label{prop:convo-facts}
\begin{enumerate}
\item If $\ell(yw) = \ell(y) + \ell(w)$, then we have isomorphisms 
\[
\dmix_{yw} \cong \dmix_y \star^B \dmix_w, \qquad \nmix_{yw} \cong \nmix_y \star^B \nmix_w
\]
in $\Dmix_B(\cB,\E)$.\label{it:convo-length}
\item We have isomorphisms
\[
\dmix_w \star^B \nmix_{w^{-1}} \cong \uuE_{\cB_e} \cong \nmix_{w^{-1}} \star^B \dmix_w
\]
in $\Dmix_B(\cB,\E)$.\label{it:convo-inv}
\end{enumerate}
\end{prop}
\begin{rmk}\label{rmk:convo-facts}
This proposition is analogous to~\cite[Facts~2.2(a) and~2.2(b)]{bbm}.  The functors $({-}) \star^B \dmix_w$ and $({-}) \star^B \nmix_w$ are denoted $R^!_w$ and $R^*_w$, respectively, in {\it loc.~cit}.  Note that~\cite[Fact~2.2(a)]{bbm} contains a misprint: it should instead assert that $R^!_{w_1} \circ R^!_{w_2} \cong R^!_{w_2w_1}$ when $\ell(w_2w_1) = \ell(w_1) + \ell(w_2)$.
\end{rmk}
\begin{proof}
For both parts, by associativity of the convolution product it is enough to consider the case where $w$ is a simple reflection $s$.  We will study convolution products with the following distinguished triangle from Lemma~\ref{lem:dt-w-ws}:
\begin{equation}\label{eqn:convo}
\dmix_s \to \uuE_{\overline{\cB_s}} \to \dmix_e\{1\} \xrightarrow{[1]}.
\end{equation}

For part~\eqref{it:convo-length}, we assume that $y$ is such that $ys > y$.  Therefore, $\pi^s_{\ddag}\dmix_y \cong \dmix_{\overline{y}}\{-1\}$ by~\eqref{eqn:pi-delta-nabla}.  Applying $\dmix_y \star^B ({-})$ to~\eqref{eqn:convo} and using Lemma~\ref{lem:convo-ewi}, we obtain a triangle $\dmix_y \star^B \dmix_s \to \pi^{s\dag} \dmix_{\overline{y}} \to \dmix_y\{ 1\}\xrightarrow{[1]}$.  This is again an instance of the distinguished triangle in Lemma~\ref{lem:dt-w-ws}, which tells us that the first term must be isomorphic to $\dmix_{ys}$, as desired.

For part~\eqref{it:convo-inv}, we apply $\nmix_s\{-1\} \star^B ({-})$ to~\eqref{eqn:convo} to obtain the distinguished triangle $\nmix_s \star^B \dmix_s\{-1\} \to \pi^{s\dag} \nmix_{\overline{s}} \to \nmix_s \xrightarrow{[1]} $.  This triangle is Verdier dual to~\eqref{eqn:convo}.  In particular, we have $\nmix_s \star^B \dmix_s \cong \nmix_e \cong \uuE_{\cB_e}$. This proves the second isomorphism. The first one can be proved similarly.
\end{proof}

\begin{prop}\label{prop:convo-exact}
Let $w \in W$.
\begin{enumerate}
\item The functors 
\[
({-}) \star^B \nmix_w, \ \nmix_w \star^B ({-}) : \Dmix_B(\cB,\E) \to \Dmix_B(\cB,\E)
\]
are right t-exact with respect to the perverse t-structure.\label{it:convo-right-exact}
\item The functors 
\[
({-}) \star^B \dmix_w, \ \dmix_w \star^B ({-}) : \Dmix_B(\cB,\E) \to \Dmix_B(\cB,\E)
\]
are left t-exact with respect to the perverse t-structure.\label{it:convo-left-exact}
\end{enumerate}
In particular, for any $w, y \in W$, $\nmix_y \star^B \dmix_w$ and $\dmix_w \star^B \nmix_y$ are perverse.
\end{prop}

This statement is analogous to~\cite[Proposition~8.2.4]{abg}.  (The proof in {\it loc.~cit.} seems to contain a misprint: it claims t-exactness properties opposite to those in the statement above.)

\begin{proof}
Proposition~\ref{prop:convo-facts} implies that $({-})\star^B \dmix_w$ (resp.~$\dmix_w \star^B ({-})$) is right adjoint to $({-})\star^B \nmix_{w^{-1}}$ (resp.~$\nmix_{w^{-1}} \star^B ({-})$).  Since a right adjoint to a right t-exact functor is left t-exact, it suffices to prove part~\eqref{it:convo-right-exact} of the proposition.

Let $s$ be a simple reflection.  We claim that for any $y \in W$, we have $\dmix_y \star^B \nmix_s \in \p\Dmix_B(\cB,\E)^{\le 0}$.  If $ys < y$, then Proposition~\ref{prop:convo-facts} tells us that $\dmix_y \star^B \nmix_s \cong \dmix_{ys}$.  On the other hand, if $ys > y$, apply $\dmix_y \star^B ({-})$ to the Verdier dual of~\eqref{eqn:convo} to obtain a triangle $\dmix_y \{-1\} \to \pi^{s\dag} \dmix_{\overline{y}} \to \dmix_y \star^B \nmix_s \xrightarrow{[1]}$.  (Here, we have used~\eqref{eqn:pi-delta-nabla} and Lemma~\ref{lem:convo-ewi}.)  The claim follows from the fact that $\pi^{s\dag}$ is right t-exact (see Corollary~\ref{cor:ps-exactness}).

Since $({-}) \star^B \nmix_s$ takes every $\dmix_y$ to an object of $\p\Dmix_B(\cB,\E)^{\le 0}$, the equivariant analogue of Proposition~\ref{prop:t-structure} tells us that $({-}) \star^B \nmix_s$ is right t-exact.  Using Propostion~\ref{prop:convo-facts} and induction on the length of $w$, we find that $({-}) \star^B \nmix_w$ is right t-exact for any $w \in W$, as desired.

The proof of right t-exactness for $\nmix_w \star^B ({-})$ is similar, but not quite identical.  As above, it is enough to prove that for any simple reflection $s$ and any $y \in W$, we have $\nmix_s \star^B \dmix_y \in \p\Dmix_B(\cB,\E)^{\le 0}$.  Also as above, if $sy < y$, then $\nmix_s \star^B \dmix_y \cong \dmix_{sy} \in \p\Dmix_B(\cB,\E)^{\le 0}$.  Suppose now that $sy > y$.  Applying $({-}) \star^B \dmix_y$ to~\eqref{eqn:convo} yields a triangle $\dmix_{sy} \to \uuE_{\overline{\cB_s}} \star^B \dmix_y \to \dmix_y\{1\} \xrightarrow{[1]}$.  This triangle shows that $\uuE_{\overline{\cB_s}} \star^B \dmix_y \in \p\Dmix_B(\cB,\E)^{\le 0}$.  Now, apply $({-}) \star^B \dmix_y$ to the Verdier dual of~\eqref{eqn:convo} to obtain a triangle $\dmix_y\{-1\} \to \uuE_{\overline{\cB_s}} \star^B \dmix_y \to \nmix_s \star^B \dmix_y \xrightarrow{[1]}$.  From the preceding observation, we conclude that $\nmix_s \star^B \dmix_y \in \p\Dmix_B(\cB,\E)^{\le 0}$, as desired.
\end{proof}

The main result of this section is the following.

\begin{thm}\label{thm:a2}
For any $I \subset S$ of finite type, the partial flag variety $\scP^I$ satisfies assumption~{\bf (A2)} of \S{\rm \ref{ss:qh-structure}}.
\end{thm}

This result applies, in particular, to the full flag variety $\cB$.

\begin{proof}
As observed in~\S\ref{ss:qh-structure} it is enough to treat the case $\E=\K$ or $\F$, which we assume from now on. Moreover, since $\D_X$ is t-exact in this case, it is enough to prove that the objects $\dmix_{\overline{w}}$ are perverse. Finally,
since the forgetful functor $\For : \Dmix_B(\scP^I,\E) \to \Dmix_{(B)}(\scP^I,\E)$ is t-exact, it is enough to prove that the objects $\dmix_{\overline{w}}$ of $\Dmix_B(\scP^I,\E)$ is perverse.

When $I$ is empty, the objects $\dmix_w \cong \dmix_w \star^B \nmix_e$ 
are perverse by Proposition~\ref{prop:convo-exact}.  Suppose now that $I$ is nonempty.  Let $w \in W^I$.
The object $\dmix_{\overline{w}}$ automatically belongs to $\p\Dmix_{B}(\scP^I,\E)^{\le 0}$, so we need only show that it lies in $\p\Dmix_{B}(\scP^I,\E)^{\ge 0}$.  Since $\pi^{I\dag}$ is t-exact and kills no nonzero perverse sheaf (see Corollaries~\ref{cor:ps-exactness} and~\ref{cor:ps-IC}), it suffices to show that $\pi^{I\dag}\dmix_{\overline{w}}$ lies in $\p\Dmix_{(B)}(\cB,\E)^{\ge 0}$.  By Lemmas~\ref{lem:ps-delta-nabla} and~\ref{lem:convo-ewi}, we have
\[
\pi^{I\dag}\dmix_{\overline{w}}
\cong \pi^{I\dag} \pi^I_* \dmix_w \cong \dmix_w \star^B \uuE_{\overline{\cB_{w_I}}}.
\]
Now $\uuE_{\overline{\cB_{w_I}}}$ is perverse by Lemma~\ref{lem:ic-smooth}.
By Proposition~\ref{prop:convo-exact}, it follows that $\dmix_w \star^B \uuE_{\overline{\cB_{w_I}}}$ lies in $\p\Dmix_{(B)}(\cB,\E)^{\ge 0}$, as desired.
\end{proof}

As an interesting special case we record the following result.

\begin{cor}
The affine flag variety and the affine Grassmannian of a reductive group satisfy assumption {\bf (A2)} of \S{\rm \ref{ss:qh-structure}}.
\end{cor}

\subsection{Complements}
\label{ss:complements}

With Theorem~\ref{thm:a2} in hand, one might pursue a more detailed study of the structure of, say, standard or tilting objects in $\Perv^\mix_{(B)}(\cB,\E)$.  Indeed, a number of facts that are well known in characteristic~$0$ or in the non-mixed setting hold in $\Perv^\mix_{(B)}(\cB,\E)$, often with the same proofs. For simplicity, in this subsection we only consider the non-equivariant setting. These results will not be used in the rest of the paper.

Observe first that the distinguished triangle in Lemma~\ref{lem:dt-w-ws} can be rearranged to give a short exact sequence:
\begin{equation}
\label{eqn:ses-pis}
\dmix_{ws}\la -1\ra \hookrightarrow \dmix_w \twoheadrightarrow \pi^{s\dag}\dmix_{\overline{w}}.
\end{equation}
This can be used to establish the following lemma, by imitating the argument of~\cite[\S2.1]{bbm} or~\cite[Lemma~4.4.7]{by}. 

\begin{lem}
\label{lem:socles}
Assume that $\E = \K$ or $\F$, and
let $w \in W$.
\begin{enumerate}
\item
There exists an embedding $\IC^\mix_e \langle -\ell(w) \rangle \hookrightarrow \dmix_w$ whose cokernel has no composition factor of the form $\IC^\mix_e \langle n \rangle$. Moreover, $\IC^\mix_e \langle -\ell(w) \rangle$ is the socle of $\dmix_w$.\label{it:socle-std}
\item
There exists a surjection $\nmix_w \twoheadrightarrow \IC^\mix_e \langle \ell(w) \rangle$ whose kernel has no composition factor of the form $\IC^\mix_e \langle n \rangle$. Moreover, $\IC^\mix_e \langle \ell(w) \rangle$ is the head of $\nmix_w$.\label{it:cosoc-costd}
\end{enumerate}
\end{lem}

Using this lemma, one can check (again assuming $\E = \K$ or $\F$, and using the reciprocity formula) that the graded ring
\[
A := \left( \bigoplus_{n \in \Z} \Hom(\cP^\mix_e, \cP^\mix_e\la n\ra) \la -n \ra \right)^{\mathrm{op}}
\]
is concentrated in even nonnegative degrees, and then, by the arguments of~\cite[\S 2.1]{bbm}, that the restriction of the exact functor
\[
\VV^\mix: \Perv^\mix_{(B)}(\cB,\E) \to A\lh\gmod
\quad\text{defined by}\quad
\VV^\mix(\cF) = \bigoplus_{n \in \Z} \Hom(\cP^\mix_e, \cF\la n\ra) \la -n \ra
\]
(where $A\lh\gmod$ is the category of finite dimensional graded $A$-modules) to the subcategory $\Tilt^\mix_{(B)}(\cB,\E)$ is fully faithful.

The next result is an analogue of a classical fact about category $\cO$.

\begin{prop}
Assume that $\E = \K$ or $\F$, and
let $w, y \in W$.  We have
\[
\dim \Hom(\dmix_w, \dmix_y \la n\ra) = 
\begin{cases}
1 & \text{if $w \le y$ and $n = \ell(y) - \ell(w)$,} \\
0 & \text{otherwise.}
\end{cases}
\]
Moreover, if $w \leq y$, then every nonzero map $\dmix_w \to \dmix_y \la \ell(y) - \ell(w) \ra$ is injective.
\end{prop}
\begin{proof}
First, using adjunction it is easily checked that $\Hom(\dmix_w, \dmix_y \la n\ra)=0$ if $w \not\leq y$. Now, assume that $w \le y$.
The description of socles in Lemma~\ref{lem:socles} implies that there is no nonzero map $\dmix_w \to \dmix_y\la n\ra$ if $n \ne \ell(y) - \ell(w)$, and that any nonzero map $\dmix_w \to \dmix_y\la \ell(y) - \ell(w)\ra$ is injective.

Let $m = \ell(y) - \ell(w)$.  
To compute $\dim \Hom(\dmix_w, \dmix_y\la m \ra)$, we proceed by induction on $\ell(y)$.  If $y= e$, then $w=e$ and the result is trivial.  Otherwise, choose a simple reflection $s$ such that $ys < y$.  If $ws < w$ as well, then applying the equivalence $({-}) \star^B \nmix_s$ and using Proposition~\ref{prop:convo-facts} we obtain $\dim \Hom(\dmix_w, \dmix_y\la m\ra) = \dim \Hom(\dmix_{ws}, \dmix_{ys}\la m\ra) = 1$.  If $ws > w$, then $\Hom(\dmix_w, \pi^{s\dag} \dmix_{\overline{y}}\la m\ra) = 0$ because, by Corollary~\ref{cor:ps-exactness}, $\IC^\mix_w$ cannot occur as a composition factor of $\pi^{s\dag}\dmix_{\overline{y}}$. Using~\eqref{eqn:ses-pis}, we therefore have $\dim \Hom(\dmix_w, \dmix_y\la m\ra) = \dim \Hom(\dmix_w, \dmix_{ys}\la m-1\ra)$, and the result follows by induction (since $w \leq ys$ under our assumptions).
\end{proof}

We conclude with a result that makes sense only when $G$ is of finite type.  In this case, let $w_0$ denote the longest element of $W$.

\begin{prop}[Geometric Ringel duality]\label{prop:ringel}
The functor 
\[
\Radon^\mix := ({-}) \star^B \dmix_{w_0} : \Dmix_{(B)}(\cB,\E) \to \Dmix_{(B)}(\cB,\E)
\]
is an equivalence of triangulated categories, with quasi-inverse $({-}) \star^B \nmix_{w_0}$. Moreover, this functor satisfies 
\[
\Radon^\mix(\nmix_w) \cong \dmix_{ww_0},
\qquad
\Radon^\mix(\cT^\mix_w) \cong \cP^\mix_{ww_0}.
\]
\end{prop}
\begin{proof}[Proof sketch]
This result is a consequence of Proposition~\ref{prop:convo-facts}; see~\cite[Proposition~2.3]{bbm} and \cite[Corollary~B.9]{ar}.
\end{proof}

\begin{rmk}
As in~\cite[\S 2.2]{bbm}, if $\E = \K$ or $\F$ one can check that we have $\VV^\mix \circ \Radon^\mix \cong \VV^\mix \circ \la -\ell(w_0) \ra$, and deduce (as in~\cite[Corollary~2.4]{bbm}) that $\VV^\mix$ is also fully faithful on projective mixed perverse sheaves.
\end{rmk}

\section{Self-duality and formality}
\label{sec:formality}

\subsection{More notation}
\label{ss:more-notation}

In this section we continue the study of the case of flag varieties of Kac--Moody groups begun in Section~\ref{sec:kac-moody}, but we assume in addition that the group is of finite type. More precisely we let $G$ be a connected complex reductive algebraic group, $B \subset G$ be a Borel subgroup and $T \subset B$ be a maximal torus. The variety $\cB:=G/B$ is the flag variety of a (finite type) Kac--Moody group, so the results of Section~\ref{sec:kac-moody} apply to this situation. We will use the obvious analogues of the categories and objects defined in that section. (We will not work with equivariant derived categories anymore.)

Let us briefly review the conventions established in~\cite[\S2.1--2.2]{ar}.  
The following assumption will be in force in this section:
\[
\text{the characteristic of $\F$ is a good prime for $G$.}
\]
We have the abelian category $\Perv_{(B)}(\cB,\E) \subset \Db_{(B)}(\cB,\E)$ of (ordinary, not mixed) perverse sheaves.  This category contains the usual menagerie of objects:
\[
\Delta_w(\E), \quad \nabla_w(\E), \quad \IC_w(\E), \quad \cT_w(\E), \quad \cP_w(\E)
\]
for $w \in W$. We also have the parity sheaf $\cE_w(\E) \in \Parity_{(B)}(\cB,\E) \subset \Db_{(B)}(\cB,\E)$, again for $w \in W$.
Let $\Tilt_{(B)}(\cB,\E) \subset \Perv_{(B)}(\cB,\E)$ be the full additive subcategory consisting of tilting sheaves.  As in~\cite[Lemma~B.5]{ar}, the natural functors
\begin{equation}\label{eqn:old-dereq}
\Kb\Tilt_{(B)}(\cB,\E) \to \Db\Perv_{(B)}(\cB,\E) \to \Db_{(B)}(\cB,\E)
\end{equation}
are equivalences of categories.

\begin{rmk}
The categories in the preceding paragraph could have been introduced in the general setting of Section~\ref{sec:mixed-der}, along with the equivalences~\eqref{eqn:old-dereq}, but they would have served no purpose: as noted in~\S\ref{ss:Dmix}, for general $X$, we have no way to relate $\Dmix_\scS(X,\E)$ to $\Db_\scS(X,\E)$.
\end{rmk}

Finally, let $\Gv$ be the Langlands dual group, with Borel subgroup $\Bv \subset \Gv$, maximal torus $\Tv \subset \Bv$, and flag variety $\cBv = \Gv/\Bv$. We assume that the system of positive roots of $(\Gv,\Tv)$ determined by $\Bv$ coincides with the system of positive coroots of $(G,T)$ determined by $B$.  In general, h\'a\v cek accents will be used to denote objects attached to $\Gv$ rather than to $G$.  For instance, $\dv_w(\E)$ is a standard object in $\Perv_{(\Bv)}(\cBv,\E)$, and $\cTv^\mix_w(\E)$ is a tilting object in $\Perv^\mix_{(\Bv)}(\cBv,\E)$.

\subsection{The functor $\nu$}
\label{ss:nu}

The main result of~\cite{ar} asserts that there exists a functor
\begin{equation}
\label{eqn:functor-nu}
\nu: \Parity_{(B)}(\cB,\E) \to \Tilt_{(\Bv)}(\cBv,\E).
\end{equation}
such that $\nu(\cE_w) \cong \cTv_{w^{-1}}$, along with an isomorphism $\nu \simto \nu \circ \{1\}$ such that the map
\begin{equation}\label{eqn:degr}
\bigoplus_{n \in \Z}\Hom(\cF, \cG\{n\}) \to \Hom(\nu \cF, \nu \cG)
\end{equation}
is an isomorphism for all $\cF,\cG$ in $\Parity_{(B)}(\cB,\E)$. (More precisely, this statement is obtained by applying~\cite[Theorem~2.1]{ar} to the group $\Gv$. In \cite{ar} we give an explicit construction of such a functor, but in this paper $\nu$ can be any functor with the above properties.) This result has strong consequences in the mixed setting that are invisible in the non-mixed world.  In this section, we exploit those consequences to prove the self-duality theorem and formality theorem for flag varieties (see \S\S\ref{ss:intro-self-duality}--\ref{ss:intro-formality}).

The functor $\nu$ of~\eqref{eqn:functor-nu} gives rise to a functor $\Dmix_{(B)}(\cB,\E)=\Kb\Parity_{(B)}(\cB,\E) \to \Kb\Tilt_{(\Bv)}(\cBv,\E)$.  Composing the latter with the equivalence $\Kb\Tilt_{(\Bv)}(\cBv,\E) \cong \Db_{(\Bv)}(\cBv,\E)$ of~\eqref{eqn:old-dereq}, we obtain a functor
\begin{equation}\label{eqn:functor-nu-extended}
\nu: \Dmix_{(B)}(\cB,\E) \to \Db_{(\Bv)}(\cBv,\E).
\end{equation}
(This use of $\nu$ should not result in any ambiguity.)  The version of $\nu$ in~\eqref{eqn:functor-nu-extended} still enjoys the property~\eqref{eqn:degr}, where now $\cF$ and $\cG$ may be arbitrary objects of $\Dmix_{(B)}(\cB,\E)$.  In addition, we have the following isomorphisms:
\[
\nu \circ \{1\} \cong \nu, \qquad
\nu \circ [1] \cong \{1\} \circ \nu, \qquad
\nu \circ \la 1\ra \cong \{1\} \circ \nu, \qquad \nu(\cE^\mix_w) \cong \cTv_{w^{-1}}.
\]

\subsection{The self-duality theorem}
\label{ss:self-duality}

We begin by calculating the value of the functor in~\eqref{eqn:functor-nu-extended} on certain special classes of objects.

\begin{lem}
\label{lem:nu-standard}
We have $\nu(\dmix_w) \cong \dv_{w^{-1}}$ and $\nu(\nmix_w) \cong \nv_{w^{-1}}$.
\end{lem}
\begin{proof}
We prove only the first isomorphism; the second one can be treated similarly.
We proceed by induction on $w$ with respect to the Bruhat order.  When $w = e$, we have $\dmix_e \cong \cE_e^\mix$ and $\dv_e = \cTv_e$, so the result is clear.

Now suppose $w > e$, and that the result is known for all $v < w$.  Let $Z_w = \overline{\cB_w} \smallsetminus \cB_w$, and let ${\check Z}_{w^{-1}} = \overline{\cBv_{w^{-1}}} \smallsetminus \cBv_{w^{-1}}$.  Recall that $\Dmix_{(B)}(Z_w,\E)$ is identified with the full triangulated subcategory of $\Dmix_{(B)}(\cB,\E)$ generated by $\{ \dmix_v\{n\} \mid v < w\}$, and similarly for $\Db_{(\Bv)}({\check Z}_{w^{-1}},\E) \subset \Db_{(\Bv)}(\cBv,\E)$.  The inductive assumption implies that 
\begin{equation}
\label{eqn:nu-Z_w}
\nu(\Dmix_{(B)}(Z_w,\E)) \subset \Db_{(\Bv)}({\check Z}_{w^{-1}},\E).
\end{equation}

Let $N$ be the cone of the canonical map $\dmix_w \to \cE_w^\mix$.  Recall that $N$ is in $\Dmix_{(B)}(Z_w,\E)$, and that $\Hom(\dmix_w, \dmix_v\{n\}[m]) = 0$ for all $v < w$ and all $n, m \in \Z$.  Now consider the distinguished triangle
\[
\nu(\dmix_w) \to \cTv_{w^{-1}} \to \nu(N) \xrightarrow{\{1\}}.
\]
According to \eqref{eqn:nu-Z_w}, we have $\nu(N) \in \Db_{(\Bv)}({\check Z}_{w^{-1}}, \E)$.  Moreover, by~\eqref{eqn:degr} and induction, we have that $\Hom(\nu(\dmix_w), \dv_{v^{-1}}[m]) = 0$ for all $v < w$ and all $m \in \Z$.  These two properties uniquely characterize the distinguished triangle whose first arrow is $\dv_{w^{-1}} \to \cTv_{w^{-1}}$, so we must have $\nu(\dmix_w) \cong \dv_{w^{-1}}$.
\end{proof}

\begin{prop}
\label{prop:nu-tilt-parity}
The functor $\nu$ induces an equivalence of categories 
\[
\bar\nu : \Tilt^\mix_{(B)}(\cB,\E) \simto \Parity_{(\Bv)}(\cBv,\E)
\]
such that $\bar\nu(\cT^\mix_w) \cong \cEv_{w^{-1}}$ and $\bar\nu \circ \la 1\ra \cong \{1\} \circ \bar\nu$.
\end{prop}

\begin{proof}
Let $w \in W$. 
Let us show that
\begin{equation}\label{eqn:nu-tilt-stalk}
\Hom \bigl( \nu(\cT^\mix_w), \nv_v \{n\} \bigr) = 
\begin{cases}
0 & \text{if} \ n \not\equiv \ell(w)-\ell(v) \pmod 2;\\
\text{a free $\E$-module} & \text{if} \ n \equiv \ell(w)-\ell(v) \pmod 2.
\end{cases}
\end{equation}
Using \eqref{eqn:degr} and Lemma \ref{lem:nu-standard}, this would follow if we knew that 
\[
\Hom(\cT^\mix_w, \nmix_{v^{-1}}\{m\} [n]) = 
\begin{cases}
0 & \text{if} \ n \not\equiv \ell(w)-\ell(v) \pmod 2;\\
\text{a free $\E$-module} & \text{if} \ n \equiv \ell(w)-\ell(v) \pmod 2.
\end{cases}
\]
Now $\nmix_{v^{-1}}\{m\}[n] = \nmix_{v^{-1}} \la -m\ra[n+m]$, and $\Hom(\cT^\mix_w, \nmix_{v^{-1}}\la -m\ra [n+m])$ clearly vanishes if $n+m \ne 0$.  When $n+m = 0$, we have
\[
\Hom(\cT^\mix_w, \nmix_{v^{-1}}\la n\ra) \cong \E^{\oplus ( \cT^\mix_w : \dmix_{v^{-1}}\la n\ra )},
\]
and this is zero unless $n \equiv \ell(w)-\ell(v) \pmod 2$ by Lemma~\ref{lem:tilt-even}. One can prove similarly that
\begin{equation}\label{eqn:nu-tilt-costalk}
\Hom \bigl( \dv_v\{n\}, \nu(\cT^\mix_w) \bigr) = 
\begin{cases}
0 & \text{if} \ n \not\equiv \ell(w)-\ell(v) \pmod 2;\\
\text{a free $\E$-module} & \text{if} \ n \equiv \ell(w)-\ell(v) \pmod 2.
\end{cases}
\end{equation}
Together,~\eqref{eqn:nu-tilt-stalk} and~\eqref{eqn:nu-tilt-costalk} say that the stalks and costalks of $\nu(\cT^\mix_w)$ are concentrated in degrees of the same parity as $\ell(w)$, and that their cohomology sheaves are free local systems. So $\nu(\cT^\mix_w)$ is a parity complex.

Thus, $\nu$ restricts to a functor $\bar\nu: \Tilt^\mix_{(B)}(\cB,\E) \to \Parity_{(\Bv)}(\cBv,\E)$.  Consider the isomorphism~\eqref{eqn:degr} for $\cF, \cG \in \Tilt^\mix_{(B)}(\cB,\E)$.  On the left-hand side of~\eqref{eqn:degr}, all summands with $n \ne 0$ vanish, so $\bar\nu$ is fully faithful.
As a consequence, the object $\bar\nu(\cT^\mix_w)$ is indecomposable. Since the standard objects appearing in a standard filtration of $\cT^\mix_w$ are supported on $\overline{\cB_w}$, with $\dmix_w$ appearing with multiplicity one, using Lemma \ref{lem:nu-standard} we deduce that $\bar\nu(\cT^\mix_w) \cong \cEv_{w^{-1}}$.

Finally, since any object of $\Parity_{(\Bv)}(\cBv,\E)$ is a direct sum of shifts of objects $\cEv_w$, we deduce that $\bar\nu$ is essentially surjective, and hence an equivalence.
\end{proof} 

The first main result of this section is the following.  Recall that the assumption that $\ell$ is good for $G$ remains in effect.

\begin{thm}[Self-duality]\label{thm:self-duality}
There exists an equivalence of triangulated categories
\[
\kappa: \Dmix_{(B)}(\cB,\E) \simto \Dmix_{(\Bv)}(\cBv,\E)
\]
satisfying
$\kappa \circ [1] \cong [1] \circ \kappa$,  $\kappa \circ \la 1 \ra \cong \{1\} \circ \kappa$, $\kappa \circ \{1\} \cong \la 1 \ra \circ \kappa$, and
\[
\kappa(\dmix_w) \cong \dv^\mix_{w^{-1}}, \quad
\kappa(\nmix_w) \cong \nv^\mix_{w^{-1}}, \quad
\kappa(\cT^\mix_w) \cong \cEv^\mix_{w^{-1}}, \quad
\kappa(\cE^\mix_w) \cong \cTv^\mix_{w^{-1}}.
\]
\end{thm}

We will prove this theorem simultaneously with Proposition~\ref{prop:degr} below. To state the latter result we choose a functor $\nuv : \Parity_{(\Bv)}(\cBv,\E) \to \Tilt_{(B)}(\cB,\E)$ which satisfies the same properties as the functor $\nu$ of
\eqref{eqn:functor-nu}, but with the roles of $G$ and $\Gv$ reversed. We denote by the same symbol the ``extended'' functor $\Dmix_{(\Bv)}(\cBv,\E) \to \Db_{(B)}(\cB,\E)$ defined as in~\eqref{eqn:functor-nu-extended}.

\begin{prop}\label{prop:degr}
Let $\mu: \Dmix_{(B)}(\cB,\E) \to \Db_{(B)}(\cB,\E)$ be the functor given by $\mu = \nuv \circ \kappa$, where $\kappa$ is as in Theorem~{\rm \ref{thm:self-duality}}.  There is an isomorphism $\mu \circ \la 1\ra \cong \mu$ such that the induced map
\begin{equation}\label{eqn:mu-degr}
\bigoplus_{n \in \Z} \Hom(\cF,\cG\la n\ra) \to \Hom(\mu \cF, \mu \cG)
\end{equation}
is an isomorphism for all $\cF,\cG \in \Dmix_{(B)}(\cB,\E)$.  Moreover, $\mu$ is t-exact and satisfies
\begin{gather*}
\mu(\dmix_w) \cong \Delta_w, \quad \mu(\nmix_w) \cong \nabla_w, \quad \mu(\IC^\mix_w) \cong \IC_w, \\
\mu(\cT^\mix_w) \cong \cT_w, \quad \mu(\cE^\mix_w) \cong \cE_w.
\end{gather*}
\end{prop}

Similarly, one can consider the functor $\muv:= \nu \circ \kappa^{-1} : D^\mix_{(\Bv)}(\cBv,\E) \to \Db_{(\Bv)}(\cBv,\E)$. This functor satisfies properties similar to those stated in Proposition~\ref{prop:degr}.
A summary of the functors and their behavior on various objects is shown in Figure~\ref{fig:functors}.

\begin{proof}[Proofs of Theorem~{\rm \ref{thm:self-duality}} and Proposition~{\rm \ref{prop:degr}}]
We define $\kappa$ to be the composition of the following equivalences:
\[
\Dmix_{(B)}(\cB,\E) \xrightarrow[\sim]{{\rm Lem.~\ref{lem:perv-dereq}}} \Kb \Tilt^\mix_{(B)}(\cB,\E)
\xrightarrow[\sim]{{\rm Prop.~\ref{prop:nu-tilt-parity}}} \Kb \Parity_{(\Bv)}(\cBv,\E) = \Dmix_{(\Bv)}(\cBv,\E).
\]
The claims about the interaction of $\kappa$ with $[1]$, $\la 1\ra$, and $\{1\}$ are clear.

The fact that $\kappa(\cT^\mix_w) \cong \cEv^\mix_{w^{-1}}$ is immediate from Proposition~\ref{prop:nu-tilt-parity}. Then the determination of $\kappa(\dmix_w)$ and $\kappa(\nmix_w)$ can be carried out using the same arguments as in the proof of Lemma~\ref{lem:nu-standard}.

Before studying $\kappa(\cE^\mix_w)$, we turn our attention to $\mu$. The isomorphism~\eqref{eqn:mu-degr} and the determination of $\mu(\dmix_w)$, $\mu(\nmix_w)$, and $\mu(\cT^\mix_w)$ all follow from what we already know about $\kappa$ and $\nuv$, using~\eqref{eqn:degr} and Lemma~\ref{lem:nu-standard} (for the group $\Gv$). Combining these facts with the description of the perverse t-structure in Proposition~\ref{prop:t-structure}, we deduce that $\mu$ is t-exact, and then that $\mu(\IC^\mix_w) \cong \IC_w$.

Next, one can use~\eqref{eqn:mu-degr} to study $\Hom(\mu(\cE^\mix_w),\nabla_v[k])$ and $\Hom(\Delta_v[k], \mu(\cE^\mix_w))$ as in the proof of Proposition~\ref{prop:nu-tilt-parity}, and deduce (using Remark~\ref{rmk:restriction-X_s}) that $\mu(\cE^\mix_w)$ is a parity complex supported on $\overline{\cB_w}$, and that its restriction to $\cB_w$ is $\uuline{\E}{}_{\cB_w}$.  The isomorphism~\eqref{eqn:mu-degr} also implies that $\End(\mu(\cE^\mix_w)) \cong \End(\cE^\mix_w)$, so $\mu(\cE^\mix_w)$ is indecomposable.  We conclude that $\mu(\cE^\mix_w) \cong \cE_w$.  The proof of Proposition~\ref{prop:degr} is now complete.

It remains to show that $\kappa(\cE^\mix_w) \cong \cTv^\mix_{w^{-1}}$.  By Proposition \ref{prop:nu-tilt-parity} applied to $\Gv$, we have $\nuv(\cTv^\mix_{w^{-1}}) \cong \cE_w$, i.e.~$\mu(\kappa^{-1} \cTv^\mix_{w^{-1}}) \cong \cE_w$. Fix such an isomorphism and an isomorphism $\mu(\cE_w^\mix) \cong \cE_w$, and consider the following instance of~\eqref{eqn:mu-degr}:
\[
\bigoplus_{n \in \Z} \Hom(\kappa^{-1} \cTv^\mix_{w^{-1}}, \cE_w^\mix \la n \ra) \simto \End(\cE_w).
\]
Since $\cE_w$ is indecomposable, the ring $\End(\cE_w)$ is local. Lifting the identity morphism through this isomorphism, we obtain a family of morphisms $\kappa^{-1} \cTv^\mix_{w^{-1}} \to \cE_w^\mix \la n \ra$ such that the sum of their images under $\mu$ is invertible. By locality, one  of these morphisms must have this property, i.e.~there exist $n \in \Z$ and $f : \kappa^{-1} \cTv^\mix_{w^{-1}} \to \cE_w^\mix \la n \ra$ whose image under $\mu$ is an isomorphism. As $\mu$ is triangulated and kills no object (by~\eqref{eqn:mu-degr}), this implies that $f$ itself is an isomorphism.  That is, $\kappa^{-1}(\cTv^\mix_{w^{-1}}) \cong \cE^\mix_w\la n\ra$.  Lastly, using the behavior of $\kappa$ on standard and costandard objects, it is easy to check that $n$ must be $0$.
\end{proof}

\begin{figure}
\[
\xymatrix@R=0.1cm@C=1cm{
& \Dmix_{(B)}(\cB,\E) \ar[ld]_-{\mu} \ar[r]^-{\kappa}_{\sim} \ar@/_2ex/[drr]_(.6){\nu} & \Dmix_{(\Bv)}(\cBv,\E) \ar[rd]^-{\muv} \ar@/^2ex/[dll]^(.6){\nuv}|(.26)\hole & \\
\Db_{(B)}(\cB,\E) & & & \Db_{(\Bv)}(\cBv,\E). \\
\Delta_w & \dmix_w \ar@{|->}[l] \ar@{|->}[r] & \dv^\mix_{w^{-1}} \ar@{|->}[r] & \dv_{w^{-1}} \\
\nabla_w & \nmix_w \ar@{|->}[l] \ar@{|->}[r] & \nv^\mix_{w^{-1}} \ar@{|->}[r] & \nv_{w^{-1}} \\
\cE_w & \cE_w^\mix \ar@{|->}[l] \ar@{|->}[r] & \cTv^\mix_{w^{-1}} \ar@{|->}[r] & \cTv_{w^{-1}} \\
\cT_w & \cT^\mix_w \ar@{|->}[l] \ar@{|->}[r] & \cEv_{w^{-1}}^\mix \ar@{|->}[r] & \cEv_{w^{-1}}.
}
\]
\caption{Functors and objects in the self-duality theorem}\label{fig:functors}
\end{figure}

As an immediate application of Theorem~\ref{thm:self-duality} and Proposition~\ref{prop:degr} we obtain the following criterion for a parity sheaf $\cE_w$ to be perverse. We will consider this question in a more systematic way in~\cite{ar3}.

\begin{cor}\label{cor:parity-tilting}
The parity sheaf $\cE_w \in \Db_{(B)}(\cB,\E)$ is perverse if and only if we have $(\cTv^\mix_{w^{-1}} : \nv^\mix_u\la n\ra) = 0$ for all $n > 0$.
\end{cor}
\begin{proof}
First we observe
that $\cE_w(\O)$ is perverse if and only if $\cE_w(\F)$ is perverse. (This follows from the facts that $\F(\cE_w(\O)) \cong \cE_w(\F)$---see \cite[Proposition~2.39]{jmw}---and that $\cE_w(\O)$ has free stalks and costalks.) On the other hand, it follows from Proposition~\ref{prop:tilting-O} that the multiplicities $(\cTv^\mix_{w^{-1}}(\E) : \nv^\mix_u(\E)\la n\ra)$ for $\E = \O$ and $\E = \F$ coincide.  In other words, the case $\E = \O$ of the corollary follows from the case $\E = \F$.

We henceforth assume that $\E = \K$ or $\F$.  Because $\cE_w$ is self-Verdier-dual, it is perverse if it just lies in $\p\Db_{(B)}(\cB,\E)^{\ge 0}$, or in other words, if $\Hom(\Delta_{u^{-1}}\{n\}, \cE_w) = 0$ for all $u \in W$ and all $n > 0$.  By Proposition~\ref{prop:degr}, the latter holds if and only if $\Hom(\dmix_{u^{-1}}\{n\},\cE^\mix_w\la m\ra) = 0$ for all $n > 0$ and all $m \in \Z$.  Next, we apply $\kappa$, and translate that condition to $\Hom(\dv^\mix_u\la n\ra, \cTv^\mix_{w^{-1}}\{ m\}) = 0$.  This $\Hom$-group always vanishes for $m \ne 0$.  When $m = 0$, it vanishes if and only if $(\cT^\mix_{w^{-1}} : \nv^\mix_u\la n\ra) = 0$, as desired.
\end{proof}

We conclude this subsection by showing that the various functors constructed above are compatible with the functors $\F({-})$ and $\K({-})$.

\begin{prop}
\label{prop:nu-F}
The functors $\nu_\E$ and $\kappa_\E$ are compatible with extension of scalars in the sense that the following diagram commutes up to isomorphisms of functors:
\[
\xymatrix@C=1.5cm{
\Dmix_{(\Bv)}(\cBv,\K) & \Dmix_{(\cBv)}(\cBv,\O) \ar[r]^-{\F(-)} \ar[l]_-{\K(-)} & \Dmix_{(\cBv)}(\cBv,\F) \\
\Dmix_{(B)}(\cB,\K) \ar[d]_-{\nu_\K} \ar[u]^-{\kappa_\K}_-{\wr} & \Dmix_{(B)}(\cB,\O) \ar[r]^-{\F(-)} \ar[l]_-{\K(-)} \ar[d]_-{\nu_\O} \ar[u]^-{\kappa_\O}_-{\wr}& \Dmix_{(B)}(\cB,\F) \ar[d]^-{\nu_\F} \ar[u]_-{\kappa_\F}^-{\wr} \\
\Db_{(\Bv)}(\cBv,\K) & \Db_{(\cBv)}(\cBv,\O) \ar[r]^-{\F(-)} \ar[l]_-{\K(-)} & \Db_{(\cBv)}(\cBv,\F).
}
\]
\end{prop}
\begin{proof}
The compatibility of the functor~\eqref{eqn:functor-nu} with extension of scalars was established in \cite[Theorem~2.1(4)]{ar}. From that, we deduce the commutativity of the bottom half of the above diagram, using~\cite[Lemma~B.6]{ar}.

That portion of the statement gives rise to a similar commutative diagram for the functor $\bar\nu$ considered in Proposition~\ref{prop:nu-tilt-parity}: 
\[
\xymatrix@C=1.5cm{
\Tilt^\mix_{(B)}(\cB,\K) \ar[d]_-{\bar\nu_\K}^-{\wr} & \Tilt^\mix_{(B)}(\cB,\O) \ar[r]^-{\F(-)} \ar[l]_-{\K(-)} \ar[d]_-{\bar\nu_\O}^-{\wr} & \Tilt^\mix_{(B)}(\cB,\F) \ar[d]^-{\bar\nu_\F}_-{\wr} \\
\Parity_{(\Bv)}(\cBv,\K) & \Parity_{(\cBv)}(\cBv,\O) \ar[r]^-{\F(-)} \ar[l]_-{\K(-)} & \Parity_{(\cBv)}(\cBv,\F).
}
\]
We now deduce the commutativity of the upper half of the diagram in the statement of the proposition, using Lemma~\ref{lem:perv-ext-scalars}.
\end{proof}

The preceding proposition immediately implies the following additional result.

\begin{cor}
\label{cor:mu-F}
The functors $\mu_\E$ are compatible with extension of scalars in the sense that the following diagram commutes up to isomorphisms of functors:
\[
\xymatrix@C=1.5cm{
\Dmix_{(B)}(\cB,\K) \ar[d]_-{\mu_\K} & \Dmix_{(B)}(\cB,\O) \ar[r]^-{\F(-)} \ar[l]_-{\K(-)} \ar[d]_-{\mu_\O} & \Dmix_{(B)}(\cB,\F) \ar[d]^-{\mu_\F} \\
\Db_{(B)}(\cB,\K) & \Db_{(B)}(\cB,\O) \ar[r]^-{\F(-)} \ar[l]_-{\K(-)} & \Db_{(B)}(\cB,\F).
}
\]
\end{cor}

\subsection{Formality}

It follows in particular from Proposition~\ref{prop:degr} that the functor $\mu$ restricts to a functor $\mu_{\mathsf{T}} : \Tilt^\mix_{(B)}(\cB,\E) \to \Tilt_{(B)}(\cB,\E)$.
The following technical lemma will be needed below.

\begin{lem}
\label{lem:mu-Tilt}
The following diagram commutes up to an isomorphism of functors:
\[
\xymatrix@C=2cm{
\Kb \Tilt^\mix_{(B)}(\cB,\E) \ar[r]^-{\sim}_{{\rm Lem.~\ref{lem:perv-dereq}}} \ar[d]_-{\Kb \mu_{\mathsf{T}}} & D^\mix_{(B)}(\cB,\E) \ar[d]^-{\mu} \\
\Kb \Tilt_{(B)}(\cB,\E) \ar[r]^-{\sim}_-{\eqref{eqn:old-dereq}} & \Db_{(B)}(\cB,\E).
}
\]
\end{lem}

\begin{proof}
Let us complete our diagram with the other equivalences in Lemma~\ref{lem:perv-dereq} and \eqref{eqn:old-dereq}:
\[
\xymatrix{
\Kb \Tilt^\mix_{(B)}(\cB,\E) \ar[r]^-{\sim} \ar[d]_-{\Kb \mu_{\mathsf{T}}} & \Db \Perv^\mix_{(B)}(\cB,\E) \ar[r]^-{\sim} \ar[d]_-{\Db \mu_{\mathsf{P}}} & D^\mix_{(B)}(\cB,\E) \ar[d]^-{\mu} \\
\Kb \Tilt_{(B)}(\cB,\E) \ar[r]^-{\sim} & \Db \Perv_{(B)}(\cB,\E) \ar[r]^-{\sim} & \Db_{(B)}(\cB,\E).
}
\]
(Here, $\mu_{\mathsf{P}} :  \Perv^\mix_{(B)}(\cB,\E) \to  \Perv_{(B)}(\cB,\E)$ is the restriction of $\mu$, an exact functor of abelian categories, and $\Db \mu_{\mathsf{P}}$ is the induced functor between bounded derived categories.) The left square in this diagram is easily seen to be commutative. The right square also commutes by~\cite[Lemma~A.7.1]{beilinson}. The claim follows.
\end{proof}

Let $\cE^\mix := \bigoplus_{w \in W} \cE^\mix_w$.  
Let us choose a bounded complex $\cF^\bullet$ of objects of $\Tilt^\mix_{(B)}(\cB, \E)$ whose image under the composition 
\[
C^{\mathrm{b}} \Tilt^\mix_{(B)}(\cB, \E) \to \Kb \Tilt^\mix_{(B)}(\cB, \E) \xrightarrow[{\rm Lem.~\ref{lem:perv-dereq}}]{\sim} D^\mix_{(B)}(\cB,\E)
\]
is $\cE^\mix$.  Consider the differential bigraded algebra (or dgg-algebra) $\underline{E}^{\bullet,\bullet}$ given by
\[
\underline{E}^{i,j} = \uHom^i(\cF^\bullet, \cF^\bullet\la -j\ra),
\]
where $\uHom^i(A^\bullet,B^\bullet) = \prod_{q-p = i} \Hom(A^p,B^q)$.  This $\Z^2$-graded $\E$-algebra is endowed with a differential (of bidegree $(1,0)$) induced by that of $\cF^\bullet$.

Next, let $\cE := \bigoplus_{w \in W} \cE_w$, and let $\cG^\bullet \in C^{\mathrm{b}} \Tilt_{(B)}(\cB,\E)$ be the image of $\cF^\bullet$ under the functor induced by $\mu_{\mathsf{T}}$. 
By Lemma~\ref{lem:mu-Tilt}, the image of this complex under the composition
\[
C^{\mathrm{b}} \Tilt_{(B)}(\cB, \E) \to \Kb \Tilt_{(B)}(\cB, \E) \xrightarrow[\eqref{eqn:old-dereq}]{\sim} \Db_{(B)}(\cB,\E)
\]
is isomorphic to $\mu(\cE^\mix) \cong \cE$.
Form the dg-algebra $E^\bullet$ given by
\[
E^i = \uHom^i(\cG^\bullet, \cG^\bullet).
\]
By Proposition~\ref{prop:degr}, $E^\bullet$ is just the dg-algebra obtained by forgetting the second grading on $\underline{E}^{\bullet,\bullet}$.  Next, form the cohomology rings
\[
\underline{\mathbf{E}} := \mathsf{H}^\bullet(\underline{E}^{\bullet,\bullet})
\qquad\text{and}\qquad
\mathbf{E} := \mathsf{H}^\bullet(E^\bullet).
\]
Again, $\underline{\mathbf{E}}$ is a bigraded ring, and $\mathbf{E}$ is obtained from it by forgetting the second grading.  It follows from the equivalence~\eqref{eqn:old-dereq} that
\[
\mathbf{E} \ \cong \ \bigoplus_{k \in \Z} \Hom_{\Db_{(B)}(\cB,\E)}(\cE, \cE\{k\})
\]
as a graded algebra.

\begin{lem}\label{lem:diagonal-cohomology}
We have $\underline{\mathbf{E}}^{i,j} = 0$ unless $i = j$.  Moreover, the ring $\mathbf{E}$ has finite global dimension.
\end{lem}
\begin{proof}
It follows from Lemma~\ref{lem:perv-dereq} that $\underline{\mathbf{E}}^{i,j} \cong \Hom_{\Dmix_{(B)}(\cB,\E)}(\cE^\mix, \cE^\mix\la -j\ra[i])$, and the latter clearly vanishes unless $i = j$.  For the second assertion, we first observe that the case $\E=\O$ follows from the case $\E=\F$, by the arguments in \cite[Lemma~5.5.3]{rsw}. So we assume that $\E=\K$ or $\F$. We have to prove that the ring
\[
\bigoplus_{k \in \Z} \Hom_{\Db_{(B)}(\cB,\E)}(\cE, \cE\{k\})
\]
has finite global dimension. Using the functor $\nu$ of~\eqref{eqn:functor-nu}, this is equivalent to proving that the ring $\End(\cTv)$ has finite global dimension, where $\cTv:=\bigoplus_{w \in W} \cTv_w$. Finally, using the Radon transform of \cite[\S 2.3]{ar}, it is enough to prove that the ring $\End(\cPv)$ has finite global dimension, where $\cPv:=\bigoplus_{w \in W} \cPv_w$.
This follows simply from the fact that all quasihereditary categories have finite cohomological dimension~\cite[Corollary~3.2.2]{bgs}.
\end{proof}

Using Lemma~\ref{lem:diagonal-cohomology} we can define
the derived category $\mathsf{dgDerf}\lh\mathbf{E}$ of finitely generated right dg-modules over the dg-algebra $\mathbf{E}$ (endowed with the trivial differential) as in \cite[\S 5.4]{rsw}. The following result is a generalization of~\cite[Theorem 5.5.8]{rsw}.  We continue to assume that $\ell$ is good for $G$.

\begin{thm}
\label{thm:formality}
There exists an equivalence of categories
\[
\Db_{(B)}(\cB,\E) \cong \mathsf{dgDerf}\lh\mathbf{E}.
\]
\end{thm}

This result implies that the hypotheses of~\cite[Theorem~1.2.1]{rsw} can be weakened.  Specifically, the ``modular Koszul duality'' relationship between Soergel's modular category $\cO$ and $\Perv_{(B)}(\cB,\F)$, proved in~\cite[Theorem~1.2.1]{rsw} for $\ell > 2\dim \cB + 1$, actually holds as soon as $\ell$ is larger than the Coxeter number for $G$.

\begin{proof}
Consider the functor $\uHom^\bullet(\cG^\bullet, {-})$ from the category of bounded complexes of objects of $\Tilt_{(B)}(\cB,\E)$ to the category of right $E^\bullet$-dg-modules. This functor descends to homotopy categories, and composing with the natural functor to the derived category $\mathsf{dgDer}\lh E^\bullet$ of right $E^\bullet$-dg-modules, we obtain a functor
\[
\Phi : \Db_{(B)}(\cB,\E) \cong \Kb \Tilt_{(B)}(\cB,\E) \to \mathsf{dgDer}\lh E^\bullet.
\]
It follows from Lemma \ref{lem:diagonal-cohomology} and a classical argument (see e.g.~\cite[Lemma 5.5.1]{rsw}) that the dg-algebra $E^\bullet$ is formal; in particular we obtain a natural equivalence of triangulated categories $\mathsf{dgDer}\lh E^\bullet \simto \mathsf{dgDer}\lh \mathbf{E}$.  Composing with $\Phi$ we obtain a natural functor
\[
\Psi : \Db_{(B)}(\cB,\E) \to \mathsf{dgDer}\lh \mathbf{E}.
\]
It is not difficult to check that this functor factors through an equivalence of triangulated categories $\Db_{(B)}(\cB,\E) \simto \mathsf{dgDerf}\lh \mathbf{E}$: see e.g.~the arguments in the proof of \cite[Theorem 5.5.8]{rsw}.
\end{proof}

\appendix
\section{Quasihereditary categories}
\label{sec:homological}

The theory of quasihereditary categories is by now quite well known, but most of the standard references work in the ungraded setting.  Here, we record the definition and a number of useful facts (mostly without proof) in the graded case.  Throughout, $\bk$ will be a field, and $\cA$ will be a finite-length $\bk$-linear abelian category.  

Assume $\cA$ is equipped with an automorphism $\la 1\ra: \cA \to \cA$.  Let $\Irr(\cA)$ be the set of isomorphism classes of irreducible objects of $\cA$, and let $\scS = \Irr(\cA)/\Z$, where $n \in \Z$ acts on $\Irr(\cA)$ by $\la n\ra$.  Assume that $\scS$ is equipped with a partial order $\le$, and that for each $s \in \scS$, we have a fixed representative simple object $\Lgr_s$. Assume also we are given, for any $s \in \scS$, objects $\dgr_s$ and $\ngr_s$, and morphisms $\dgr_s \to \Lgr_s$ and $\Lgr_s \to \ngr_s$. For $\scT \subset \scS$, we denote by $\cA_{\scT}$ the Serre subcategory of $\cA$ generated by the objects $\Lgr_t \la n \ra$ for $t \in \scT$ and $n \in \Z$. We write $\cA_{\leq s}$ for $\cA_{\{t \in \scS \mid t \leq s\}}$, and similarly for $\cA_{<s}$.

\begin{defn}\label{defn:qhered}
The category $\cA$ is said to be \emph{graded quasihereditary} if the following conditions hold:
\begin{enumerate}
\item The set $\scS$ is finite.\label{it:qh-def-fin}
\item For each $s \in \scS$, we have 
\[
\Hom(\Lgr_s,\Lgr_s \la n \ra) = \begin{cases}
\bk & \text{if $n=0$;} \\
0 & \text{otherwise.}
\end{cases}
\]
 \label{it:qh-def-split}
\item For any $\scT \subset \scS$ closed (for the order topology) and such that $s \in \scT$ is maximal, $\dgr_s \to \Lgr_s$ is a projective cover in $\cA_{\scT}$ and $\Lgr_s \to \ngr_s$ is an injective envelope in $\cA_{\scT}$.
\item The kernel of $\dgr_s \to \Lgr_s$ and the cokernel of $\Lgr_s \to \ngr_s$ belong to $\cA_{<s}$.
\label{it:qh-def-ker}
\item We have $\Ext^2(\dgr_s, \ngr_t\la n\ra) = 0$ for all $s, t \in \scS$ and $n \in \Z$.\label{it:qh-def-ext2}
\end{enumerate}
\end{defn}

If $\cA$ satisfies Definition~\ref{defn:qhered},
the objects $\dgr_s\la n\ra$ are called \emph{standard objects}, and the objects $\ngr_s\la n\ra$ are called \emph{costandard objects}.

\begin{defn}\label{defn:tilting}
Let $X$ be an object in a graded quasihereditary category.  We say that $X$ \emph{admits a standard} (resp.~\emph{costandard}) \emph{filtration} if there exists a filtration $F_\bullet X$ such that each $\Gr^F_i X$ is isomorphic to some $\dgr_s\la n\ra$ (resp.~$\ngr_s\la n\ra$).  We say that $X$ is \emph{tilting} if it admits both a standard and a costandard filtration.
\end{defn}

Every direct factor of a tilting object is tilting.
If $X$ admits a standard filtration $F_\bullet X$, it is easy to see that for fixed $s \in \scS$ and $n \in \Z$, the number of $i$ such that $\Gr^F_i X \cong \dgr_s\la n\ra$ is equal to $\dim \Hom(X, \ngr_s\la n\ra)$, and is hence independent of the choice of filtration $F_\bullet X$.  We introduce the notation
\[
(X : \dgr_s\la n\ra) := \dim \Hom(X, \ngr_s\la n\ra)
\]
for objects with a standard filtration.  Similarly, if $Y$ admits a costandard filtration, then the multiplicity of a given $\ngr_s\la n\ra$ in any costandard filtration is equal to
\[
(Y : \ngr_s\la n\ra) := \dim \Hom(\dgr_s\la n\ra, Y).
\]
The following is a graded analogue of~\cite[Theorem~3.2.1 and Corollary~3.2.2]{bgs}.  We omit the proof.

\begin{thm}\label{thm:bgs}
Let $\cA$ be a graded quasihereditary category.  Then $\cA$ has enough projectives, and every projective admits a standard filtration. Moreover, if $P^\gr_s$ denotes a projective cover of $\Lgr_s$, then $(P^\gr_s : \dgr_t\la n\ra) = [\ngr_t\la n\ra : \Lgr_s]$ for all $t \in \scS$.  Lastly, every object in $\cA$ admits a finite projective resolution.
\end{thm}

As explained in~\cite[Corollary~3]{ringel}, the fact that projectives admit standard filtrations can be used to show that
\begin{equation}\label{eqn:ext-std-costd}
\Ext^k(\dgr_s, \ngr_t\la n\ra) = 0 \qquad\text{for all $s, t \in \scS$, $n \in \Z$, and $k \ge 1$.}
\end{equation}

The following statement gives the well-known classification of indecomposable tilting objects.  For a proof, see~\cite[Proposition~2]{ringel}.

\begin{prop}\label{prop:tilt-class}
For each $s \in \scS$, there is a unique indecomposable tilting object $T_s^\gr$ contained in $\cA_{\le s}$ and satisfying $(T_s^\gr : \dgr_s) = (T_s^\gr : \ngr_s) = 1$.  Moreover, every indecomposable tilting object is isomorphic to some $T_s^\gr \la n\ra$.
\end{prop}

Finally, we have two useful derived-equivalence results.

\begin{lem}\label{lem:tilt-equiv}
Let $\cA$ be a graded quasihereditary category, and let $\Tilt(\cA) \subset \cA$ be the full additive subcategory consisting of tilting objects.  The natural functor $\Kb\Tilt(\cA) \to \Db(\cA)$ is an equivalence of categories.
\end{lem}
\begin{proof}
It follows from~\eqref{eqn:ext-std-costd} that if $T$ and $T'$ are tilting objects, then $\Ext^k_\cA(T,T') = 0$ for all $k > 0$. The rest of the argument follows~\cite[Proposition~1.5]{bbm}.
\end{proof}

\begin{lem}\label{lem:qher-dereq}
Let $\cA$ be a graded quasihereditary category.  Suppose that $\cA$ is the heart of a bounded t-structure on a triangulated category $\cT$, and that $\la 1 \ra$ is the restriction of an automorphism of $\cT$ (denoted similarly).  Suppose that the following conditions hold.
\begin{enumerate}
\item $\cT$ is a full subcategory of the bounded derived category of an abelian category, or of the bounded homotopy category of some additive category.
\item In $\cT$, we have $\Hom(\dgr_s, \ngr_t\la n\ra[k]) = 0$ whenever $k > 0$.
\end{enumerate}
Then there is an equivalence of categories $\Db(\cA) \simto \cT$.
\end{lem}
\begin{proof}
The first condition allows us to construct a ``realization functor'' $\Db(\cA) \to \cT$, using either~\cite[\S3.1]{bbd} or~\cite[\S2.5]{ar:kdsf}.  Then, using the second condition, one can repeat the proof of Lemma~\ref{lem:tilt-equiv} to deduce that in the diagram $\Kb\Tilt(\cA) \to \Db(\cA) \to \cT$, the first functor and the composition of the two functors are equivalences of categories. It follows that $\Db(\cA) \to \cT$ is an equivalence as well.
\end{proof}


\end{document}